\documentclass{amsart}
\usepackage{macros}
\standardsettings

\newcommand{\Av}{\mathrm{Av}}

\newcommand{\tmax}{{\typewriter max}}
\newcommand{\tmin}{{\typewriter min}}
\newcommand{\tprod}{{\typewriter prod}}
\newcommand{\tsum}{{\typewriter sum}}
\newcommand{\tlcm}{{\typewriter lcm}}

\newcommand{\Hmax}{H_{\tmax}}
\newcommand{\Hmin}{H_{\tmin}}
\newcommand{\Hprod}{H_{\tprod}}
\newcommand{\Hlcm}{H_{\tlcm}}

\newcommand{\Htheta}{H_{\Theta}}
\renewcommand{\Rplus}{[0,\infty)}
\newcommand{\tplus}{[t_0,\infty)}
\newcommand{\Tplus}{[t_1,\infty)}

\draftfalse

\begin{document}
\title[Unconventional height functions in Diophantine approximation]{Unconventional height functions in simultaneous Diophantine approximation}

\authorlior\authordavid


\begin{Abstract}
Simultaneous Diophantine approximation is concerned with the approximation of a point $\mathbf x\in\mathbb R^d$ by points $\mathbf r\in\mathbb Q^d$, with a view towards jointly minimizing the quantities $\|\mathbf x - \mathbf r\|$ and $H(\mathbf r)$. Here $H(\mathbf r)$ is the so-called ``standard height'' of the rational point $\mathbf r$. In this paper the authors ask: What changes if we replace the standard height function by a different one? As it turns out, this change leads to dramatic differences from the classical theory and requires the development of new methods. We discuss three examples of nonstandard height functions, computing their exponents of irrationality as well as giving more precise results. A list of open questions is also given.
\end{Abstract}
\maketitle

Fix $d\geq 1$, and for each function $\Theta:\N^d\to\N$ let $\Htheta:\Q^d\to\N$ be defined by the formula
\[
\Htheta\left(\frac{p_1}{q_1},\ldots,\frac{p_d}{q_d}\right) = \Theta(q_1,\ldots,q_d).
\]
Here we assume that $p_1/q_1,\ldots,p_d/q_d\in\Q$ are given in reduced form. The function $\Htheta$ will be called a \emph{height function} on $\Q^d$.

Classical simultaneous Diophantine approximation is concerned with the \emph{standard height function} $\Hlcm$, where $\tlcm:\N^d\to\N$ is the least common multiple function. Historically, this height function and its variations and generalizations (see e.g. \cite[\6VIII.5-6]{Silverman}) have played a major role in modern mathematics, not only in Diophantine approximation but also in the theories of projective varieties and elliptic curves.\Footnote{Although what we call here the ``standard height function'' is the most commonly considered height function in the field of Diophantine approximation, a slightly different height function is considered to be standard in other areas of number theory. Namely, if a rational $\pp/q\in\Q^d$ is in reduced form, then many number theorists, motivated by projective geometry, define the height of $\pp/q$ to be the number $\max(|p_1|,\ldots,|p_d|,q)$ rather than $q$. The two height functions agree on rationals in the unit cube $[0,1]^d$, as well as agreeing up to a multiplicative error term on bounded subsets of $\R^d$, so the difference is rarely significant.} The standard height function has been treated as the natural choice for a height function on $\Q^d$, to the point where no other choices were even considered. One reason for the historical emphasis on the standard height function is its connection to the lattice $\Z^d$; specifically; given $\rr\in\Q^d$, $\Hlcm(\rr)$ is the smallest number $q$ such that $\rr = \pp/q$ for some $\pp\in\Z^d$. This way of interpreting $\Hlcm$ lends itself more easily to generalizations to projective varieties and algebraic number fields; cf. \cite[Remark VIII.5.5]{Silverman}. The connection to lattices also induces a connection between the Diophantine approximation based on this height function and the dynamics of the homogeneous space $\SL_{d + 1}(\R)/\SL_{d + 1}(\Z)$; cf. \cite[Theorem 8.5]{KleinbockMargulis}.

The aim of this paper is to broaden the viewpoint of simultaneous Diophantine approximation by considering alternative height functions. Specifically, we will consider the height functions $\Hmax$, $\Hmin$, and $\Hprod$ defined by the maximum, minimum, and product functions $\tmax,\tmin,\tprod:\N^d\to\N$.\Footnote{It has been pointed out to us that there are definitions of the term ``height function'' according to which $\Hmin$ is not a height function, since its sublevelsets $\{\Hmin \leq q\}$ ($q\in\N$) are not discrete (i.e. it does not satisfy the the Northcott property on compact sets). However, for our purposes it is not important (except for one place where we must be slightly careful, see Footnote \ref{footnotedirichletdef} below) whether the sublevelsets of $\Hmin$ are discrete, and we feel that the role played by $\Hmin$ in this paper is sufficiently ``height-function-like'' (in a Diophantine approximation sense) to justify the use of the terminology.} Although these height functions are not as related to the lattice $\Z^d$ (but see the Remark after Theorem \ref{theoremminprod} for a relation between the height functions $\Hprod$ and $\Hlcm$ based on the Segre embedding), in a certain sense they are more natural than $\Hlcm$, since the functions $\tmax$, $\tmin$, and $\tprod$ are monotonic whereas $\tlcm$ is not. Thus the study of these alternative height functions will be based not as much on the study of lattices, but will take a more ``component-wise'' approach.

The authors devote a section to analyzing a certain class of functions, the class of \emph{recursively integrable functions} (denoted $\RR$), which is used in the proof of one of the main theorems. The class $\RR$ is contained in the class of integrable functions, and is similar to it in some ways. However, unlike the class of integrable functions, the class $\RR$ is not closed under either addition or scalar multiplication. Nevertheless, there are many functions $f_2$ with the property that for every $f_1\in\RR$, we have $f_1 + f_2\in\RR$.


{\bf Acknowledgements.} The first-named author was supported in part by the Simons Foundation grant \#245708. The second-named author was supported in part by the EPSRC Programme Grant EP/J018260/1. The authors thank the referees of previous versions of this paper for their comments which helped us to write an introduction more accessible to a general audience. We thank the referee of the current version for helpful comments.


{\bf Convention 1.} For $\alpha\geq 0$, we let $\psi_\alpha(q) = q^{-\alpha}$.

{\bf Convention 2.} Given $\Theta:\N^d\to\N$ and $(q_i)_{i = 1}^d \in \N^d$, we will write
\[
\Theta_{i = 1}^d q_i := \Theta(q_1,\ldots,q_d).
\]

{\bf Convention 3.} The symbols $\lesssim$, $\gtrsim$, and $\asymp$ will denote multiplicative asymptotics. For example, $A\lesssim_K B$ means that there exists a constant $C > 0$ (the \emph{implied constant}), depending only on $K$, such that $A\leq CB$.

{\bf Convention 4.} In this paper ``increasing'' means ``nondecreasing'' and ``decreasing'' means ``nonincreasing'', unless the word ``strictly'' is added.

{\bf Convention 5.} The symbol $\triangleleft$ will be used to indicate the end of a nested proof.

\section{Main results}

Throughout, $d\geq 1$ is fixed, and $\|\cdot\|$ denotes the max norm on $\R^d$. Note that if $d = 1$, then $\Hlcm = \Hmax = \Hmin = \Hprod = H_0$, where $H_0(p/q) = q$.

We begin by recalling Dirichlet's theorem:

\begin{theorem*}[Dirichlet's Approximation Theorem] 

For each $\xx \in \R^d$, and for any $Q \in \N$, 
there exists $\pp/q\in \Q^d$ with $1\leq q \leq Q^d$ such that
\[
\left\| \xx - \frac{\pp}{q} \right\| < \frac{1}{qQ}\cdot
\]
\end{theorem*}

\begin{corollary*}[Dirichlet's Corollary]
For every $\xx \in \R^d\butnot\Q^d$, 
\[
\left\| \xx - \frac{\pp}{q}\right\| < \frac{1}{q^{1 + 1/d}} \text{ for infinitely many } \frac{\pp}{q} \in \Q^d.
\]
Equivalently,
\begin{equation}
\label{dirichletscorollary}
\| \xx - \rr_n \| < \psi_{1 + 1/d}\circ\Hlcm(\rr_n) \text{ for some sequence } \Q^d\ni \rr_n\to \xx.\footnotemark
\end{equation}
\end{corollary*}
\Footnotetext{There is a subtle distinction here: the existence of infinitely many rational points satisfying a given inequality, versus the existence of a sequence of rational points satisfying the given inequality and tending to the given point $\xx$. This distinction is important in what follows because otherwise the function $\Hmin$ would behave pathologically, since there exists a bounded region containing infinitely many points $\rr\in\Q^d$ satisfying $\Hmin(\rr) = 1$. But we do not want to say that such a sequence is a sequence of ``approximations'' of a given point $\xx$ unless the sequence actually converges to the point $\xx$ (which could only happen if $\xx$ has an integer coordinate). \label{footnotedirichletdef}}

In what follows, we consider analogues of Dirichlet's Corollary when $\Hlcm$ is replaced by one of the three height functions $\Hmax$, $\Hmin$, and $\Hprod$.

\subsection{Exponents of irrationality}
Before getting down to the details of our main theorems, we first consider ``coarse'' analogues of Dirichlet's Corollary. Specifically, we determine what the appropriate analogue of the exponent $1 + 1/d$ which appears in the formula \eqref{dirichletscorollary} should be for our nonstandard height functions. More precisely:

\begin{definition*}
Given a height function $H:\Q^d\to\N$ and a point $\xx\in\R^d\butnot\Q^d$, the \emph{exponent of irrationality} of $\xx$ is
\[
\omega_H(\xx) = \liminf_{\substack{\rr\in\Q^d \\ \rr\to\xx}} \frac{-\log\|\xx - \rr\|}{\log H(\rr)}
= \lim_{\epsilon \to 0} \inf_{\substack{\rr\in\Q^d \\ \|\xx - \rr\| \leq \epsilon}} \frac{-\log\|\xx - \rr\|}{\log H(\rr)}\cdot
\]
Equivalently, $\omega_H(\xx)$ is the supremum of all $\alpha\geq 0$ such that
\[
\left\|\xx - \rr_n\right\| < \psi_\alpha\circ H(\rr_n) \text{ for some sequence } \Q^d \ni \rr_n \to \xx.
\]
The \emph{exponent of irrationality} of the height function $H$ is the number
\[
\omega_d(H) = \inf_{\xx\in\R^d\butnot\Q^d} \omega_H(\xx).
\]
\end{definition*}

We observe that Dirichlet's Corollary implies that $\omega_d(\Hlcm) \geq 1 + 1/d$. In fact, the reverse inequality is true (and well-known):
\[
\omega_d(\Hlcm) = 1 + 1/d.
\]
This means that $1 + 1/d$ is the ``best exponent'' that can be put into formula \eqref{dirichletscorollary}.

We are now ready to state the following theorem regarding exponents of irrationality:

\begin{theorem}[Exponents of irrationality of $\Hmax$, $\Hmin$, and $\Hprod$]
\label{theoremexponents}
\begin{align} \label{exponents1}
\omega_d(\Hmax) &= \frac{d}{(d - 1)^{(d - 1)/d}} \hspace{.3 in} \text{if $d\geq 2$}\\ \label{exponents2}
\omega_d(\Hmin) &= 2\\ \label{exponents3}
\omega_d(\Hprod) &= \frac 2d\cdot
\end{align}
\end{theorem}

\begin{remark*}
The inequalities $\tmin\leq\tprod^{1/d}\leq\tmax\leq\tlcm\leq\tprod$ automatically imply that
\[
\omega_d(\Hprod) \leq \omega_d(\Hlcm) \leq \omega_d(\Hmax) \leq d \phantom{\cdot} \omega_d(\Hprod) \leq\omega_d(\Hmin).
\]
Theorem \ref{theoremexponents} shows that when $d\geq 3$, all inequalities are strict except the last. (When $d = 2$, the third inequality is also not strict.) It is also interesting to note that $\lim_{d\to\infty}\omega_d(\Hmax) = 1 = \lim_{d\to\infty}\omega_d(\Hlcm)$, so the second inequality is asymptotically an equality.
\end{remark*}

\subsection{More precise results}
We now prepare to state our main theorems. These theorems will answer the question of what the appropriate analogue of the function $\psi_{1 + 1/d}$ should be for our nonstandard height functions. More precisely:

\begin{definition*}
Given a height function $H:\Q^d\to\N$, a function $\psi:\N\to(0,\infty)$, and a point $\xx\in\R^d$, let
\begin{equation}
\label{Cdef}
C_{H,\psi}(\xx) = \liminf_{\substack{\rr\in\Q^d \\ \rr\to\xx}} \frac{\|\xx - \rr\|}{\psi\circ H(\rr)}\cdot
\end{equation}
Equivalently, $C_{H,\psi}(\xx)$ is the infimum of all $C\geq 0$ such that
\[
\|\xx - \rr_n\| < C \psi\circ H(\rr_n) \text{ for some sequence $\Q^d\ni \rr_n\to \xx$.}
\]
A function $\psi$ will be called \emph{Dirichlet} on $\R^d$ with respect to the height function $H$ if $C_{H,\psi}(\xx) < \infty$ for all $\xx\in\R^d\butnot\Q^d$, \emph{uniformly Dirichlet} if $\sup_{\R^d\butnot\Q^d} C_{H,\psi} < \infty$, and \emph{optimally Dirichlet} if $\psi$ is Dirichlet and $C_{H,\psi}(\xx) > 0$ for at least one $\xx\in\R^d\butnot\Q^d$. (This terminology originally appeared in \cite{FishmanSimmons5}.) 
\end{definition*}

We observe that Dirichlet's Corollary implies that the function $\psi_{1 + 1/d}$ is uniformly Dirichlet on $\R^d$ with respect to the height function $\Hlcm$, and in fact that
\[
C_{\Hlcm,\psi_{1 + 1/d}}(\xx) \leq 1 \all \xx\in\R^d\butnot\Q^d.
\]
In fact, the function $\psi_{1 + 1/d}$ is optimally Dirichlet on $\R^d$ with respect to the height function $\Hlcm$, due to the existence of so-called \emph{badly approximable vectors}, i.e. vectors $\xx\in\R^d\butnot\Q^d$ for which $C_{\Hlcm,\psi_{1 + 1/d}}(\xx) > 0$. Roughly, the statement that $\psi_{1 + 1/d}$ is optimally Dirichlet should be interpreted as meaning that in formula \eqref{dirichletscorollary}, the function $\psi_{1 + 1/d}$ cannot be improved by more than a multiplicative constant. This interpretation was made rigorous in \cite[Theorem 2.6 and Proposition 2.7]{FishmanSimmons5}.

\begin{example*}
The function $\psi_2(q) = q^{-2}$ is uniformly and optimally Dirichlet on $\R$ with respect to the height function $H_0$. This fact may be equivalently expressed as follows:
\begin{itemize}
\item[(i)] ($\psi_2$ is uniformly Dirichlet) There exists $C > 0$ such that for all $x\in\R$, there exist infinitely many $p/q\in\Q$ such that $|x - p/q| \leq C q^{-2}$.
\item[(ii)] (Optimality) There exist $x\in\R$ and $\epsilon > 0$ such that $|x - p/q| \geq \epsilon q^{-2}$ for all but finitely many $p/q\in\Q$.
\end{itemize}
\end{example*}


\begin{remark*}
We will sometimes deal with functions $\psi$ which are not defined for all natural numbers, but only for sufficiently large numbers. In this case, the formula \eqref{Cdef} may be interpreted as referring to an arbitrary extension of $\psi$ to $\N$; it is clear that the precise nature of the extension does not matter.
\end{remark*}


Given a height function $H \in \{\Hmax,\Hmin,\Hprod\}$ and $d\geq 1$, we may now ask the following questions:

\begin{itemize}
\item[1.] Is there an optimally Dirichlet function on $\R^d$ with respect to $H$?
\item[2.] If so, what is it?\Footnote{Technically, there may be more than one optimally Dirichlet function, as shown in \cite[Remark 2.11]{FishmanSimmons5}. However, in this paper we are really only interested in \emph{Hardy $L$-functions} (which will be defined shortly), and for these functions, there is up to a multiplicative constant at most one optimally Dirichlet function (again see \cite[Remark 2.11]{FishmanSimmons5}).}
\item[3.] If not, can one give a criterion for determining whether or not a given function is Dirichlet?
\end{itemize}

It turns out that to answer these questions, we must consider two cases. The first case is when either $H \in \{\Hmin,\Hprod\}$ or $d \leq 2$. In this case, the situation is similar to the situation for the height function $\Hlcm$: there is a uniformly and optimally Dirichlet function, and it comes from the class of power law functions $(\psi_\alpha)_{\alpha\geq 0}$. Precisely:

\begin{theorem}
\label{theoremminprod}
Fix $\Theta\in\{\tmax,\tmin,\tprod\}$, and if $\Theta = \tmax$ assume that $d \leq 2$. Then the function
\[
\psi_{\omega_d(\Htheta)}(q) = \begin{cases}
q^{-2} & \Theta = \tmax,\tmin\\
q^{-2/d} & \Theta = \tprod
\end{cases}
\]
is uniformly and optimally Dirichlet on $\R^d$ with respect to the height function $\Htheta$.
\end{theorem}

\begin{remark*}
The case $\Theta = \tprod$ of Theorem \ref{theoremminprod} can be reformulated as a theorem about \emph{intrinsic} Diophantine approximation (see e.g. \cite{FKMS2}) using the \emph{standard} height function $\Hlcm$ on the variety $M_d = \Phi_d(\R^d) \subset \R^{2^d - 1}$, where
\[
\Phi_d(x_1,\ldots,x_d) = \left({\prod}_{i\in S} x_i\right)_{\smallemptyset \neq S \subset \{1,\ldots,d\}}
\]
is (the affinization of) the Segre embedding. This is because for every rational $\rr\in\Q^d$, we have $\Hprod(\rr) = \Hlcm\circ\Phi_d(\rr)$. In the terminology of \cite{FKMS2}, the reformulated theorem states that the function $\psi(q) = q^{-2/d}$ is an optimal Dirichlet function for the Diophantine triple $(M_d,\Q^{2^d - 1}\cap M_d,\Hlcm)$. (It is uniformly Dirichlet on compact subsets of this triple.) The special case $d = 2$ follows from \cite[Theorems 4.5 and 5.1]{FKMS2} using the fact that $M_2$ is a quadric hypersurface; cf. \cite[Remark 8.1]{FKMS2}.
\end{remark*}

In the second case, namely when $H = \Hmax$ and $d \geq 3$, the situation is much different. Specifically, when $d\geq 3$ the height function $\Hmax$ has the following unexpected property: It possesses no ``reasonable'' optimally Dirichlet function. To state this precisely, we need to define the class of functions that we consider to be reasonable. A \emph{Hardy $L$-function} is a function which can be expressed using only the elementary arithmetic operations $+,-,\times,\div$, exponents, logarithms, and real-valued constants, and which is well-defined on some interval of the form $(t_0,\infty)$.\Footnote{Hardy $L$-functions were defined by G. H. Hardy and were originally called \emph{logarithmico-exponential functions}; see \cite[\63]{Hardy2}.} For example, for any $C,\alpha\geq 0$ the function
\[
\psi(q) = q^{-\alpha + C/\log^2\log(q)}
\]
is is a Hardy $L$-function. We have the following:
\begin{theorem}
\label{theoremnooptimaldirichlet}
Suppose $d\geq 3$. Then no Hardy $L$-function is optimally Dirichlet on $\R^d$ with respect to the height function $\Hmax$.
\end{theorem}

\begin{remark*}
The class of Hardy $L$-functions includes almost all functions that one naturally encounters in dealing with ``analysis at infinity'', except for those with oscillatory behavior.
\end{remark*}

This answers question 1 above, so we would like next to answer question 3. Namely, given $d\geq 3$ and a Hardy $L$-function $\psi$, how does one determine whether or not $\psi$ is Dirichlet on $\R^d$ with respect to $\Hmax$? Our final theorem (Theorem \ref{theoremHmax}) will be a complete answer to this question. However, since it is complicated to state, we approach this theorem by degrees. As a first approximation we give the following corollary, which considers the case of a single error term added to the function $\psi_{\omega_d(\Hmax)}$:

\begin{corollary*}[of Theorem \ref{theoremHmax}]
\label{corollaryerrorterm}
Suppose $d\geq 3$. For each $C > 0$ let
\begin{equation}
\label{psiC}
\psi(q) = q^{-\omega_d(\Hmax) + C/\log^2\log(q)}.
\end{equation}
Then $\psi$ is (non-optimally) Dirichlet on $\R^d$ with respect to $\Hmax$ if and only if
\begin{equation}
\label{C>}
C > \frac{d\gamma_d \log^2(\gamma_d)}{8},
\end{equation}
where $\gamma_d = (d - 1)^{1/d} > 1$.
\end{corollary*}



In particular, letting $C = 0$, we see that the function $\psi_{\omega_d(\Hmax)}$ is not Dirichlet on $\R^d$ with respect to $\Hmax$.

This corollary now provides us with motivation to state our final theorem. Let $\psi$ be the function defined by \eqref{psiC} when $C = d\gamma_d\log^2(\gamma_d)/8$. We know that $\psi$ is not Dirichlet (on $\R^d$ with respect to $\Hmax$), but that for any function of the form $\phi_\epsilon(q) = q^{\epsilon/\log^2\log(q)}$, the product $\phi_\epsilon\psi$ is Dirichlet. This suggests that there is a function $\phi$ which grows more slowly than any $\phi_\epsilon$ such that the product $\phi\psi$ is still Dirichlet. What function can we multiply by? As it turns out, if
\[
\phi(q) = q^{C/[\log^2\log(q)\log^2\log\log(q)]},
\]
then $\phi\psi$ is Dirichlet if and only if \eqref{C>} holds. At this point it is clear that this line of questioning can be pursued \emph{ad infinitum}, leading to the following:

\begin{theorem}
\label{theoremHmax}
Suppose that $d\geq 3$. Then for each $N\geq 1$ and $C \geq 0$, the function
\[
\psi_{N,C}(q) = q\wedge\left(-\omega_d(\Hmax) + \frac{d\gamma_d\log^2(\gamma_d)}{8}
\left[\sum_{n = 2}^N \prod_{i = 2}^n \left(\frac{1}{\log^{(i)}(q)}\right)^2 + C\prod_{i = 2}^{N + 1}\left(\frac{1}{\log^{(i)}(q)}\right)^2 \right] \right)
\]
is (non-optimally) Dirichlet on $\R^d$ with respect to $\Hmax$ if and only if $C > 1$. Here $\gamma_d = (d - 1)^{1/d}$ as before, and $\log^{(i)}$ denotes the $i$th iterate of the logarithm function. If $N = 1$, then the first summation is equal to $0$ by convention.
\end{theorem}

The earlier corollary is precisely the special case $N = 1$ of Theorem \ref{theoremHmax}.

\begin{remark*}
It may not be entirely obvious that Theorem \ref{theoremHmax} is a complete answer to question 3 in the case of Hardy $L$-functions. Nevertheless, it is. Precisely: If $\psi$ is a Hardy $L$-function, then there exist $N\geq 1$ and $C\geq 0$ such that comparing $\psi$ with $\psi_{N,C}$ together with Theorem \ref{theoremHmax} allow one to determine whether or not $\psi$ is Dirichlet on $\R^d$ with respect to $\Hmax$. For a proof of this, see Proposition \ref{propositionLfunctions}.

For a version of Theorem \ref{theoremHmax} which goes slightly beyond Hardy $L$-functions, allowing $\psi$ to be a member of any Hardy field which contains the exponential and logarithm functions and is closed under composition, see Proposition \ref{propositionHmax}.
\end{remark*}

\subsection{Techniques}

The main technique of this paper is to generalize the correspondence between the continued fraction expansion of an irrational number and its Diophantine properties into higher dimensions. This is done by introducing the notion of a \emph{data progression} corresponding to an irrational vector $\xx$, which is a mathematical object that encodes information about the continued fraction expansions of all of the coordinates of $\xx$. The Diophantine properties of $\xx$ can then be related to properties of the corresponding data progression. For more details see \6\ref{subsectiondata}.

In the case of the height function $\Hmax$, this correspondence translates the question of which functions are Dirichlet into a question about whether data progressions satisfying certain inequalities exist. We answer this question by converting it into a question about whether certain differential equations have nonnegative solutions, leading to the concept of a \emph{recursively integrable function}. This concept is interesting in its own right and we study it in detail in Section \ref{sectionrecint}. In particular we give a complete characterization of which Hardy $L$-functions are recursively integrable (Proposition \ref{propositionLfunctions}), which leads to the characterization of which functions are Dirichlet described above in Theorem \ref{theoremHmax}.

\subsection{Summary of the paper}
Section \ref{sectionpreliminaries} contains preliminary results which are used in the proofs of our main theorems. In Section \ref{sectionminprod} we prove Theorem \ref{theoremminprod}, as well as demonstrating formulas \eqref{exponents2} and \eqref{exponents3}. Section \ref{sectioninterlude} provides a motivation for the first formula of Theorem \ref{theoremexponents} without giving a rigorous proof. Section \ref{sectionrecint} is devoted to defining and analyzing the class of recursively integrable functions, a class which is used in the proof of Theorems \ref{theoremnooptimaldirichlet} and \ref{theoremHmax}. In Section \ref{sectionHmax} we prove Theorems \ref{theoremnooptimaldirichlet} and \ref{theoremHmax}, as well as demonstrating formula \eqref{exponents1}. Finally, a list of open questions is given in Section \ref{sectionopenquestions}.

\draftnewpage
\section{Preliminaries}
\label{sectionpreliminaries}

\subsection{Lemmas concerning continued fractions}
We begin our preliminaries with two lemmas concerning continued fractions. The first states that for $x\in\R$, the convergents of the continued fraction expansion of $x$ provide the best approximations to $x$ as long as one is willing to accept a multiplicative error term.\Footnote{Here ``best approximations'' means ``best approximations of the first kind'' in the language of \cite[p.24]{Khinchin_book}. Note that if no error term is allowed, then best approximations of the first kind must be \emph{intermediate fractions} (cf. \eqref{intermediate}), but they are not necessarily convergents.} Hence the Diophantine properties of $x$ essentially depend only on the denominators of these convergents. The second states that given any sequence of numbers increasing fast enough, there is a number $x$ such that the denominators of the convergents of the continued fraction expansion of $x$ are equal up to an asymptotic to the elements of this sequence. Together, the two lemmas say that from a (sufficiently coarse) Diophantine point of view, the properties of a number can be encoded by an increasing sequence of integers.

\begin{remark*}
This subsection is mostly interesting if $x$ is an irrational number. However, since the implied constants are supposed to be independent of $x$, the results are nontrivial even when $x$ is rational.
\end{remark*}

\begin{lemma}
\label{lemmaconvergents}
Fix $x\in\R$, and let $(p_n/q_n)_0^N$ be the convergents of the continued fraction expansion of $x$ (so that $N = \infty$ if and only if $x\notin\Q$). Then for every $p/q\in\Q$, there exists $n\in\N$ so that
\[
q\gtrsim q_n \text{ and } \left|x - \frac pq\right| \gtrsim \left|x - \frac{p_n}{q_n}\right|
\]
(cf. Convention 3).
\end{lemma}
Before we begin the proof, we recall \cite[Theorem 1]{Khinchin_book} that if $(a_n)_0^N$ are the partial quotients of the continued fraction expansion of $x$, then
\begin{align} \label{recursion1}
p_n &= a_n p_{n - 1} + p_{n - 2}\\ \label{recursion2}
q_n &= a_n q_{n - 1} + q_{n - 2}
\end{align}
for all $n\geq 1$. Here we use the convention that $p_{-1} = 1$ and $q_{-1} = 0$. In particular, the sequence $(q_n)_0^N$ is strictly increasing and satisfies $q_n \asymp a_n q_{n - 1}$. We recall also \cite[Theorems 9 and 13]{Khinchin_book} that for all $0\leq n < N$,
\begin{equation}
\label{errorapprox}
\left|x - \frac{p_n}{q_n}\right| \asymp \frac{1}{q_n q_{n + 1}}\cdot
\end{equation}

\begin{proof}
Consider the set $S = \{p'/q'\in\Q : q' \leq q\}$, and let $p'/q'\in S$ be chosen to minimize $|x - p'/q'|$. Then $q'\leq q$ and $|x - p'/q'| \leq |x - p/q|$, so we may without loss of generality assume that $p/q = p'/q'$. In this case, $p/q$ is a \emph{best approximation of the first kind} in the sense of \cite[p.24]{Khinchin_book}. By \cite[Theorem 15]{Khinchin_book}, $p/q$ is an \emph{intermediate fraction} in the sense of \cite[p.14]{Khinchin_book}, i.e.
\begin{equation}
\label{intermediate}
\frac{p}{q} = \frac{a p_{n - 1} + p_{n - 2}}{a q_{n - 1} + q_{n - 2}}
\end{equation}
for some $1\leq n\leq N$ and $1\leq a\leq a_n$. We consider two cases separately:
\begin{itemize}
\item Case 1: $a\geq a_n/2$. In this case,
\[
2q \geq a_n q_{n - 1} + q_{n - 2} = q_n.
\]
On the other hand, by \cite[Theorem 17]{Khinchin_book}, $p_n/q_n$ is a best approximation of the second kind, and thus also a best approximation of the first kind. Since $q\leq q_n$, this gives
\[
\left|x - \frac{p}{q}\right| > \left|x - \frac{p_n}{q_n}\right|,
\]
completing the proof in this case.
\item Case 2: $1\leq a < a_n/2$. In this case, since $p/q$ lies on the same side of $x$ as $p_n/q_n$ (cf. \cite[Theorem 4]{Khinchin_book} and \cite[Lemma on p.14]{Khinchin_book}), we have
\begin{align*}
\left|x - \frac{p}{q}\right| &\geq \left|\frac{p_n}{q_n} - \frac pq\right|\\
&= \left| \frac{a_n p_{n - 1} + p_{n - 2}}{a_n q_{n - 1} + q_{n - 2}} - \frac{a p_{n - 1} + p_{n - 2}}{a q_{n - 1} + q_{n - 2}}\right|\\
&= \frac{a_n - a}{[a_n q_{n - 1} + q_{n - 2}][a q_{n - 1} + q_{n - 2}]} & \text{(cf. \cite[Theorem 2]{Khinchin_book})}\\
&\geq \frac{a_n/2}{q_n^2}
\asymp \frac{1}{q_{n - 1} q_n} \asymp \left|x - \frac{p_{n - 1}}{q_{n - 1}}\right|.
\end{align*}
Since $q\geq q_{n - 1}$, this completes the proof in this case.
\qedhere\end{itemize}
\end{proof}

\begin{lemma}
\label{lemmaabstractconvergents}
Let $(\w q_n)_0^N$ be a (finite or infinite) sequence satisfying $\w q_{n + 1}\geq 2 \w q_n$ and $\w q_0 = 1$. Then there exists $x\in\R$ so that if $(p_n/q_n)_0^N$ are the convergents of the continued fraction expansion of $x$, then
\begin{equation}
\label{qNbounds}
\frac{1}{2}\w q_n \leq q_n \leq \w q_n \all n\in\N.
\end{equation}
\end{lemma}
\begin{proof}
The proof will proceed by recursively defining a sequence of integers $(a_n)_1^N$ and then letting $x$ be the unique number in $(0,1)$ whose partial quotients are given by $(a_n)_1^N$. Note that once this process is completed, for every $1\leq M\leq N$ the value of $q_M$ can be computed from \eqref{recursion2} using only the data points $(a_n)_1^M$ together with the initial values $q_{-1} = 0$, $q_0 = 1$. Thus in our recursive step, once we have defined $(a_n)_1^M$, we may treat $(q_n)_1^M$ as also defined.

Fix $1\leq M\leq N$, and suppose that the values $(a_n)_1^{M - 1}$ have been fixed, and that the resulting values $(q_n)_1^{M - 1}$ all satisfy \eqref{qNbounds}. In particular, when $n = M - 1$, \eqref{qNbounds} holds. (If $M = 1$, this is due to the assumption on $\w q_0$ rather than to the induction hypothesis.) Let $a_M$ be the largest integer $a\geq 1$ such that $a q_{M - 1} + q_{M - 2} \leq \w q_M$. Such an integer exists because
\[
\w q_M \geq 2\w q_{M - 1} \geq 2 q_{M - 1} \geq q_{M - 1} + q_{M - 2}.
\]
Let $q_M$ be given by \eqref{recursion2}. Then
\[
q_M\leq \w q_M \leq (a_M + 1) q_{M - 1} + q_{M - 2}
\leq 2 (a_M q_{M - 1} + q_{M - 2}) = 2 q_M,
\]
i.e. \eqref{qNbounds} holds when $n = M$. This completes the recursive step.
\end{proof}

\subsection{Data progressions}
\label{subsectiondata}
Fix $d\geq 1$. In the previous subsection, we learned how the Diophantine properties of an irrational number $x$ are encoded in the sequence of denominators of the convergents of the continued fraction expansion of $x$. Continuing with this theme, given an irrational point $\xx\in\R^d\butnot\Q^d$ we would like to find a structure which encodes the Diophantine properties of $\xx$. It turns out that the appropriate structure for this encoding is given by the following definition:

\begin{definition}
\label{definitiondataprogression}
Let $\Delta = (A_k,i_k)_{k = 1}^\infty$ be a pair of sequences, so that $A_k\in\R$ and $i_k\in\{1,\ldots,d\}$ for all $k\in\N$. Assume that $\{i_k : k\in\N\} = \{1,\ldots,d\}$. For each $i = 1,\ldots,d$ and $k$ sufficiently large, let
\begin{align*}
\ell(i,k) &:= \max\{k' < k : i_{k'} = i\}\\
b_k^{(i)} &:= A_{\ell(i,k) + 1}.
\end{align*}
Equivalently, the sequence  $\Big(\Delta_k := (b_k^{(i)})_{i = 1}^d\Big)_{k = 1}^\infty$ may be defined via the recursive formula
\begin{equation}
\label{bkirec}
b_{k + 1}^{(i)} = \begin{cases}
A_{k + 1} & \text{if } i = i_k\\
b_k^{(i)} & \text{if } i \neq i_k
\end{cases}.
\end{equation}
We say that $\Delta$ is a \emph{$d$-dimensional data progression} if the following hold:
\begin{itemize}
\item[(I)] For all $k$ sufficiently large,
\begin{equation}
\label{dataprog1}
b_{k + 1}^{(i_k)} > b_k^{(i_k)}.
\end{equation}
\item[(II)] The sequence $(\max(\Delta_k))_{k = 1}^\infty$ is unbounded.
\end{itemize}
Given $\Xi:\Rplus^d \to\Rplus$ and $\Psi:\Rplus\to\R$ we write
\[
C_{\Xi,\Psi}(\Delta) = \liminf_{k\to\infty} \left(\Psi\left( \Xi_{i = 1}^d b_k^{(i)}\right) - b_k^{(i_k)} - b_{k + 1}^{(i_k)}\right).
\]
\end{definition}

\begin{remark*}
In the sequel, the notation introduced in this definition will be used without comment.
\end{remark*}

\begin{remark*}
A pair of sequences $\Delta = (A_k,i_k)_{k = 1}^\infty$ is a one-dimensional data progression if and only if $i_k = 1$ for all $k$, and the sequence $(A_k)_{k = 1}^\infty$ is increasing and tends to infinity. The canonical example is the sequence $(q_k)_{k = 1}^\infty$ of denominators of convergents of an irrational number $x \in \R \butnot \Q$.
\end{remark*}


\begin{lemma}
\label{lemmacondition}
Fix $\Theta:[1,\infty)^d\to[1,\infty)$ and $\psi:[1,\infty)\to(0,\infty)$. Let $\Xi = \log\Theta\exp$ and let $\Psi = -\log\psi\exp$. Suppose that $\Xi$ and $\Psi$ are uniformly continuous and coordinatewise increasing.
\begin{itemize}
\item[(i)] For each $\xx\in\R^d\butnot\Q^d$, there exists a $d$-dimensional data progression $\Delta$ such that
\begin{equation}
\label{condition1}
C_{\Htheta,\psi}(\xx) \lesssim \exp C_{\Xi,\Psi}(\Delta).
\end{equation}
\item[(ii)] Conversely, for each $d$-dimensional data progression $\Delta$, there exists $\xx\in\R^d\butnot\Q^d$ such that
\begin{equation}
\label{condition2}
C_{\Htheta,\psi}(\xx) \gtrsim_{\psi,\Theta} \exp C_{\Xi,\Psi}(\Delta).
\end{equation}
\end{itemize}
In particular
\[
\sup_{\R^d\butnot\Q^d} C_{\Htheta,\psi} \asymp_{\psi,\Theta} \exp \sup_\Delta C_{\Xi,\Psi}(\Delta),
\]
where the supremum is taken over all $d$-dimensional data progressions $\Delta$.
\end{lemma}

\begin{remark*}
The maps $\xx\mapsto \Delta$ and $\Delta\mapsto \xx$ implicitly described in parts (i) and (ii) of Lemma \ref{lemmacondition}, respectively, are in fact independent of $\Theta$ and $\psi$, as can be easily seen from the proof of Lemma \ref{lemmacondition}. On an intuitive level these maps are ``rough inverses'' of each other, but we do not make this rigorous.
\end{remark*}

\begin{remark}
\label{remarkcondition}
If $\Theta \in\{\tmax,\tmin,\tprod\}$, then $\Xi\in\{\tmax,\tmin,\tsum\}$ is uniformly continuous and coordinatewise increasing. If $\psi$ is a Hardy $L$-function whose decay is no faster than polynomial, then $\Psi$ is uniformly continuous and increasing (Lemma \ref{lemmahardy3}). Thus for the situations considered in this paper, the hypotheses of Lemma \ref{lemmacondition} will be immediately satisfied.
\end{remark}

\ignore{
}
\begin{proof}[Proof of \text{(i)}]
Fix $\xx\in\R^d\butnot\Q^d$, and for each $i = 1,\ldots,d$, let $\big(p_n^{(i)}/q_n^{(i)}\big)_{n = 1}^{N_i}$ be the convergents of the continued fraction expansion of $x_i$. Here $N_i\in\N\cup\{\infty\}$, with $N_i = \infty$ for at least one $i$. Let $E_i = \{1,\ldots,N_i - 1\}$ if $N_i \in\N$, and $E_i = \N$ if $N_i = \infty$. Let $E = \{(n,i) : i = 1,\ldots,d, \; n\in E_i\}$, and define a map $f:E\to\N$ by letting $f(n,i) = q_n^{(i)}q_{n + 1}^{(i)}$. Let $\big((m_k,i_k)\big)_{k = 1}^\infty$ be an indexing of $E$ such that the map $k\mapsto f(m_k,i_k)$ is increasing. Then for each $k\in\N$, let
\[
A_{k + 1} = \log(q_{m_k + 1}^{(i_k)}),
\]
and let $\Delta = (A_k,i_k)_{k = 1}^\infty$. Then $b_k^{(i_k)} = \log(q_{m_k}^{(i_k)})$ and $b_{k + 1}^{(i_k)} = \log(q_{m_k + 1}^{(i_k)})$. It follows immediately that $\Delta$ is a $d$-dimensional data progression. To demonstrate \eqref{condition1}, let
\begin{align*}
L(i,k) &= \min\{k'\in\N: k'\geq k, i_{k'} = i\}\\
n(i,k) &= m_{L(i,k)} = m_{\ell(i,k)} + 1,
\end{align*}
so that
\[
b_k^{(i)} = \log(q_{n(i,k)}^{(i)}).
\]
Now
\begin{align*}
\min_{i = 1}^d\left(q_{n(i,k)}^{(i)}q_{n(i,k) + 1}^{(i)}\right) &= \min_{i = 1}^d f(n(i,k),i)\\
&= \min_{i = 1}^d f(m_{L(i,k)},i_{L(i,k)}) \noreason\\
&= f(m_k,i_k) \since{$L(i_k,k) = k$, and $L(i,k) \geq k$ for all $i$}\\
&= q_{m_k}^{(i_k)} q_{m_k + 1}^{(i_k)}\\
&= \exp(b_k^{(i_k)} + b_{k + 1}^{(i_k)}).
\end{align*}
Let $\rr_k = \Big(p_{n(i,k)}^{(i)} / q_{n(i,k)}^{(i)}\Big)_{i = 1}^d$. Then
\begin{align*}
C_{\Htheta,\psi}(\xx) &\leq \liminf_{k\to\infty} \frac{\|\xx - \rr_k\|}{\psi\circ\Htheta(\rr_k)}\\
&\asymp \liminf_{k\to\infty} \max_{i = 1}^d \frac{1}{q_{n(i,k)}^{(i)}q_{n(i,k) + 1}^{(i)}}\frac{1}{\psi\circ\Htheta(\rr_k)} \by{\eqref{errorapprox}}\\
&= \liminf_{k\to\infty} \frac{1}{\exp(b_k^{(i_k)} + b_{k + 1}^{(i_k)})}\frac{1}{\psi\left(\Theta_{i = 1}^d q_{n(i,k)}^{(i)}\right)} = \exp C_{\Xi,\Psi}(\Delta). 
\noreason&&\qedhere\end{align*}
\end{proof}

\begin{proof}[Proof of \text{(ii)}]
Let $\Delta = (A_k,i_k)_{k = 1}^\infty$ be a $d$-dimensional data progression. For each $i = 1,\ldots,d$, define an increasing sequence $(k(i,n))_{n = 0}^{N_i}$ recursively: Let $k(i,0)$ be large enough so that $b_{k(i,0)}^{(i)}$ is defined. Now fix $n\geq 0$, and suppose that $k(i,n)$ has been defined. Let $k(i,n + 1)$ be the smallest value of $k$ such that
\[
b_k^{(i)} \geq b_{k(i,n)}^{(i)} + \log(2)
\]
if such a value exists; otherwise let $N_i = n$. Then by Lemma \ref{lemmaabstractconvergents}, there exists $x_i\in\R$ satisfying
\[
q_n^{(i)} \asymp \exp(b_{k(i,n)}^{(i)}) \all 1\leq n\leq N_i,
\]
where $(p_n^{(i)}/q_n^{(i)})_{n = 1}^{N_i}$ are the convergents of the continued fraction expansion of $x_i$. By (II) of the definition of a data progression, we have $N_i = \infty$ for at least one $i$ and thus $\xx := (x_1,\ldots,x_d)\notin\Q^d$. We will demonstrate \eqref{condition2}. Fix $\rr\in\Q^d$. For each $i = 1,\ldots,d$, by Lemma \ref{lemmaconvergents} there exists $n_i = n_i(\rr)$ such that $H_0(r_i) \gtrsim q_{n_i}^{(i)}$ and $|x_i - r_i| \gtrsim |x - p_{n_i}^{(i)} / q_{n_i}^{(i)}|$. Let
\begin{align*}
k_i &= k_i(\rr) = k(i,n_i(\rr) + 1) - 1\\
k &= k(\rr) = \min_{i = 1}^d k_i(\rr),
\end{align*}
so that
\[
b_{k(i,n_i)}^{(i)} \leq b_{k_i}^{(i)} \leq b_{k(i,n_i)}^{(i)} + \log(2).
\]
Here the understanding is that if $n_i = N_i$, then $k_i = \infty$ and $b_{k_i}^{(i)} = \lim_{k\to\infty} b_k^{(i)}$. Then
\[
H_0(r_i) \gtrsim q_{n_i}^{(i)} \asymp \exp(b_{k(i,n_i)}^{(i)}) \asymp \exp(b_{k_i}^{(i)}) \geq \exp(b_k^{(i)}).
\]
Using the fact that $\Xi$ and $\Psi$ are uniformly continuous and coordinatewise increasing, we deduce that
\[
\psi\circ\Htheta(\rr) = \psi\left(\Theta_{i = 1}^d H_0(r_i)\right)
\lesssim_{\psi,\Theta} \psi\left(\Theta_{i = 1}^d\exp(b_k^{(i)})\right).
\]
On the other hand, for each $i$ such that $k_i\neq \infty$ we have 
\[
|x_i - r_i| \gtrsim \left|x_i - \frac{p_{n_i}^{(i)}}{q_{n_i}^{(i)}}\right| \asymp \frac{1}{q_{n_i}^{(i)} q_{n_i + 1}^{(i)}} \asymp \frac{1}{\exp(b_{k(i,n_i)}^{(i)} + b_{k(i,n_i + 1)}^{(i)})} \asymp \frac{1}{\exp(b_{k_i}^{(i)} + b_{k_i + 1}^{(i)})}\cdot
\]
Since $i_{k_i} = i \all i$, we have $k_{i_k} = k$. Thus
\[
\|\xx - \rr\| \geq |x_{i_k} - r_{i_k}| \gtrsim \frac{1}{\exp(b_k^{(i_k)} + b_{k + 1}^{(i_k)})}\cdot
\]
Combining, we have
\[
\frac{\|\xx - \rr\|}{\psi\circ\Htheta(\rr)}
\gtrsim_{\psi,\Theta} \frac{1}{\exp(b_k^{(i_k)} + b_{k + 1}^{(i_k)})}\frac{1}{\psi\left(\Theta_{i = 1}^d\exp(b_k^{(i)})\right)}\cdot
\]
Let $(\rr_j)_1^\infty$ be a sequence in $\Q^d$ along which the liminf in \eqref{Cdef} is achieved. Since $\|\xx - \rr_j\| \to 0$, it follows that for all $i = 1,\ldots,d$, we have $n_i(\rr_j)\to \infty$ and thus $k_i(\rr_j)\to\infty$. So $k(\rr_j)\to\infty$, and thus
\begin{align*}
C_{\Htheta,\psi}(\xx)
&=_{\phantom{\psi,\Theta}} \lim_{j\to\infty} \frac{\|\xx - \rr_j\|}{\psi\circ H_\Theta(\rr_j)}\\
&\gtrsim_{\psi,\Theta} \liminf_{k\to\infty}  \frac{1}{\exp(b_k^{(i_k)} + b_{k + 1}^{(i_k)})}\frac{1}{\psi\left(\Theta_{i = 1}^d\exp(b_k^{(i)})\right)}
= \exp C_{\Xi,\Psi}(\Delta).
\qedhere\end{align*}
\end{proof}

\draftnewpage
\section{Proof of Theorem \ref{theoremminprod} and formulas \eqref{exponents2}, \eqref{exponents3}}
\label{sectionminprod}

We begin by reformulating Theorem \ref{theoremminprod} using Theorem \ref{theoremexponents}:

\begin{proposition}
\label{propositionminprod}
Fix $d\geq 1$ and $\Theta\in\{\tmax,\tmin,\tprod\}$, and if $\Theta = \tmax$ assume that $d \leq 2$. Let
\begin{equation}
\label{betaddef}
\beta_d = \begin{cases}
2 & \Theta = \tmax,\tmin\\
2/d & \Theta = \tprod
\end{cases}\cdot
\end{equation}
Then $\psi_{\beta_d}$ is uniformly and optimally Dirichlet on $\R^d$ with respect to the height function $\Htheta$.
\end{proposition}

Proving this reformulation is sufficient to prove Theorem \ref{theoremminprod}. Indeed, Proposition \ref{propositionminprod} immediately implies that $\omega_d(\Htheta) = \beta_d$; replacing $\beta_d$ by $\omega_d(\Htheta)$ in Proposition \ref{propositionminprod} yields Theorem \ref{theoremminprod}.

Proposition \ref{propositionminprod} also implies \eqref{exponents2} and \eqref{exponents3}, and the case $d = 2$ of \eqref{exponents1}.

\begin{remark*}
The case $d = 1$ of Proposition \ref{propositionminprod} merely states that $\psi_2$ is uniformly and optimally Dirichlet on $\R$ with respect to the standard height function $H_0$. Thus, in the proof we may assume $d\geq 2$.
\end{remark*}

\begin{proof}[Proof of Uniform Dirichletness]
By Lemma \ref{lemmacondition}, it suffices to show that
\[
\sup_\Delta C_{\Xi,\Psi}(\Delta) \leq 1,
\]
where $\Xi = \log\Theta\exp$, $\Psi = -\log\psi_{\beta_d}\exp$, and the supremum is taken over $d$-dimensional data progressions $\Delta$. By contradiction suppose that $C_{\Xi,\Psi}(\Delta) > 1$ for some $d$-dimensional data progression $\Delta = (A_k,i_k)_{k = 1}^\infty$. Then for all $k$ sufficiently large, we have
\begin{equation}
\label{contradictionhypothesis4}
b_k^{(i_k)} + b_{k + 1}^{(i_k)} \leq \beta_d\operatorname\Xi_{i = 1}^d b_k^{(i)} - 1.
\end{equation}
Let $\Var(\Delta)$ and $\Av(\Delta)$ denote the variance and mean (average) of a $d$-tuple $\Delta$, respectively. Let $K = \{k\in\N : \max(\Delta_{k + 1}) > \max(\Delta_k)\}$.

\ignore{
Let $C$ denote the implied constant of \eqref{contradictionhypothesis3}; we claim that
\[
\sup_{\R^d\butnot\Q^d} C_{\Hmax,\psi} \leq e^{C + 1},
\]
demonstrating that $\psi$ is uniformly Dirichlet. By contradiction, suppose that $C_{\Hmax,\psi}(\xx) > e^{C + 1}$ for some $\xx\in\R^d$. Let $\psi = \psi_{\beta_d}$ and $t = e^{C + 1}$. Then by Lemma \ref{lemmacondition}, there exist $d$ increasing sequences $\big((b_k^{(i)})_{k = 1}^\infty\big)_{i = 1}^d$ satisfying (I) and (II) of Lemma \ref{lemmacondition}. In particular, \eqref{contradictionhypothesis3} becomes
\begin{equation}
\label{contradictionhypothesis4}
b_k^{(i_k)} + b_{k + 1}^{(i_k)} \leq -\log\psi\exp\left(\operatorname\Xi_{i = 1}^d b_k^{(i)}\right) - 1.
\end{equation}
Let $\Var(\Delta)$ and $\Av(\Delta)$ denote the variance and mean (average) of a data set $\Delta$, respectively. For each $k$, let $\Delta_k$ denote the data set $(b_k^{(i)})_{i = 1}^d$. Let $K = \{k\in\N : \max(\Delta_{k + 1}) > \max(\Delta_k)\}$.
}

\begin{claim}
\label{claimvardec}
We have
\begin{align} \label{vardec1}
\Var(\Delta_{k + 1}) &\leq \Var(\Delta_k) \hspace{0.9 in} \all k\in\N\\ \label{vardec2}
\Var(\Delta_{k + 1}) &\leq \Var(\Delta_k) - 1/\max(4,d) \all k\in K.
\end{align}
\end{claim}
The proof is divided into two cases: either $\Theta\in \{\tmin,\tprod\}$, or $\Theta = \tmax$ and $d = 2$.
\begin{subproof}[Proof if $\Theta\in \{\tmin,\tprod\}$]
To begin with, we observe that
\begin{equation}
\label{DeltaVar}
\begin{split}
\Var(\Delta_{k + 1}) - \Var(\Delta_k) &\leq \frac{1}{d}\sum_{i = 1}^d \left(b_{k + 1}^{(i)} - \Av(\Delta_k)\right)^2 - \Var(\Delta_k)\\
&= \frac{1}{d}\left[\left(b_{k + 1}^{(i_k)} - \Av(\Delta_k)\right)^2 - \left(b_k^{(i_k)} - \Av(\Delta_k)\right)^2\right].
\end{split}
\end{equation}
Now by \eqref{contradictionhypothesis4}, we have
\[
b_k^{(i_k)} + b_{k + 1}^{(i_k)} \leq 2\Av(\Delta_k) - 1.
\]
(If $\Theta = \tprod$, then this equation is simply a reformulation of \eqref{contradictionhypothesis4}; if $\Theta = \tmin$, it follows from the fact that $\min(\Delta_k) \leq \Av(\Delta_k)$.) Rearranging gives
\begin{equation}
\label{12penalty}
\Av(\Delta_k) \geq \frac{b_k^{(i_k)} + b_{k + 1}^{(i_k)}}{2} + \frac{1}{2}\cdot
\end{equation}
By \eqref{dataprog1}, the above equation implies that
\[
|b_{k + 1}^{(i_k)} - \Av(\Delta_k)| \leq |b_k^{(i_k)} - \Av(\Delta_k)|.
\]
Combining with \eqref{DeltaVar} completes the proof of \eqref{vardec1}. Now suppose that $k\in K$, and observe that $\Av(\Delta_k) \leq \max(\Delta_k) < \max(\Delta_{k + 1}) = b_{k + 1}^{(i_k)}$. Combining with \eqref{12penalty} yields
\[
|b_{k + 1}^{(i_k)} - \Av(\Delta_k)| \leq |b_k^{(i_k)} - \Av(\Delta_k)| - 1,
\]
and thus
\[
|b_{k + 1}^{(i_k)} - \Av(\Delta_k)|^2 \leq |b_k^{(i_k)} - \Av(\Delta_k)|^2 - 1.
\]
Combining with \eqref{DeltaVar} gives \eqref{vardec2}.
\end{subproof}
\begin{subproof}[Proof if $\Theta = \tmax$ and $d = 2$]
In this case, \eqref{contradictionhypothesis4} becomes
\[
b_k^{(i_k)} + b_{k + 1}^{(i_k)} \leq 2 \max(b_k^{(i_k)},b_k^{(j_k)}) - 1,
\]
where $j_k$ satisfies $\{i_k,j_k\} = \{1,2\}$. Combining with \eqref{dataprog1} gives $b_k^{(i_k)} < b_k^{(j_k)}$, and so rearranging gives
\begin{equation}
\label{12penaltycase2}
b_{k + 1}^{(j_k)} = b_k^{(j_k)} \geq \frac{b_k^{(i_k)} + b_{k + 1}^{(i_k)}}{2} + \frac12.
\end{equation}
By \eqref{dataprog1}, the above equation implies that
\[
|b_{k + 1}^{(i_k)} - b_{k + 1}^{(j_k)}| = |b_{k + 1}^{(i_k)} - b_k^{(j_k)}| \leq |b_k^{(i_k)} - b_k^{(j_k)}|,
\]
demonstrating \eqref{vardec1}. Now suppose that $k\in K$, and observe that $b_k^{(j_k)} = \max(\Delta_k) < \max(\Delta_{k + 1}) = b_{k + 1}^{(i_k)}$. Combining with \eqref{12penaltycase2} gives
\[
|b_{k + 1}^{(i_k)} - b_{k + 1}^{(j_k)}| \leq |b_k^{(i_k)} - b_k^{(j_k)}| - 1,
\]
and thus
\[
|b_{k + 1}^{(i_k)} - b_{k + 1}^{(j_k)}|^2 \leq |b_k^{(i_k)} - b_k^{(j_k)}|^2 - 1.
\]
Since $\Var(\Delta_k) = (1/4)|b_k^{(i_k)} - b_k^{(j_k)}|^2$, this equation is equivalent to \eqref{vardec2}.
\end{subproof}
To complete the proof of Proposition \ref{propositionminprod}, observe that $K$ is infinite by (II) of Definition \ref{definitiondataprogression}. Thus, it follows from Claim \ref{claimvardec} that $\Var(\Delta_k) \to -\infty$. But this contradicts the fact that the variance of a data set is always nonnegative.
\end{proof}
\begin{proof}[Proof of Optimality]
Let $x_1,\ldots,x_d\in\R$ be badly approximable numbers, and let $\xx = (x_1,\ldots,x_d)$. We claim that $C_{\Htheta,\psi_{\beta_d}}(\xx) > 0$, demonstrating the optimality of $\psi_{\beta_d}$. Indeed, for each $\rr\in\Q^d$,
\[
\|\xx - \rr\| = \max_{i = 1}^d \left|x_i - r_i\right| \gtrsim_\xx \max_{i = 1}^d \frac{1}{H^2(r_i)} =  \frac{1}{\Hmin^2(\rr)} \geq \frac{1}{\Hprod^{2/d}(\rr)} \geq \frac{1}{\Hmax^2(\rr)}\cdot
\]
Thus $\|\xx - \rr\| \gtrsim_\xx \psi_{\beta_d}\circ\Htheta(\rr)$, which implies the desired result.
\end{proof}

\draftnewpage
\section{Interlude: Motivation for the value of $\omega_d(\Hmax)$}
\label{sectioninterlude}
Before jumping into the proof of Theorems \ref{theoremnooptimaldirichlet} and \ref{theoremHmax}, in this section we try to motivate the formula \eqref{exponents1}. Our approach is as follows: The notion of a ``data progression'' is very broad, but it is natural to expect that ``worst-case-scenario'' data progressions will behave somewhat regularly. In fact, we will prove a rigorous version of this assertion in Section \ref{sectionHmax}. But for now, let's just see what happens if we restrict our attention to data progressions which behave regularly.

\begin{definition*}
A data progression $\Delta$ is \emph{periodic} if the map $k\mapsto i_k$ is periodic of order $d$, and \emph{geometric} if $A_k = \gamma^k$ for some $\gamma > 1$. The number $\gamma$ is called the \emph{mutliplier}.
\end{definition*}

\begin{remark*}
If a data progression is periodic, then the map $\{1,\ldots,d\}\ni k\mapsto i_k$ must be a permutation.
\end{remark*}

\begin{remark*}
It is shown in Section \ref{sectionHmax} that to determine which functions $\psi$ are Dirichlet on $\R^d$ with respect to $\Hmax$, it is sufficient to consider data progressions which are eventually periodic (Claim \ref{claimrotating}) and asymptotically geometric (Claim \ref{claimtkgammad}).
\end{remark*}

\begin{lemma}
\label{lemmaperiodicgeometric}
Let $\Delta$ be a periodic geometric $d$-dimensional data progression of multiplier $\gamma$. Fix $\alpha\geq 0$, and let $\Psi_\alpha(b) = \alpha b$. Then
\[
C_{\tmax,\Psi_\alpha}(\Delta) = \begin{cases}
-\infty & \text{if } \gamma + \gamma^{-(d - 1)} > \alpha\\
0 & \text{if } \gamma + \gamma^{-(d - 1)} = \alpha\\
\infty & \text{if } \gamma + \gamma^{-(d - 1)} < \alpha
\end{cases}.
\]
\end{lemma}
\begin{proof}
Since $\Delta$ is periodic, we have $\{b_k^{(i)} : i = 1,\ldots,d\} = \{A_{k - j} : j = 0,\ldots,d - 1\}$, $b_k^{(i_k)} = A_{k - d + 1}$, and $b_{k + 1}^{(i_k)} = A_{k + 1}$. Thus
\begin{align*}
C_{\tmax,\Psi_\alpha}(\Delta) &= \liminf_{k\to\infty}\left(\alpha\max_{j = 0}^{d - 1} A_{k - j} - A_{k - d + 1} - A_{k + 1}\right)\\
&= \liminf_{k\to\infty}\left(\alpha\gamma^k - \gamma^{k - d + 1} - \gamma^{k + 1}\right)\\
&= \liminf_{k\to\infty}\left(\alpha - \gamma^{-(d - 1)} - \gamma\right)\gamma^k.
\end{align*}
Since $\gamma^k\to\infty$, this completes the proof.
\end{proof}
Fix $\alpha\geq 0$. From Lemma \ref{lemmacondition}, we know that $\psi_\alpha$ is Dirichlet on $\R^d$ with respect to $\Hmax$ if and only if $C_{\tmax,\Psi_\alpha}(\Delta) < \infty$ for every $d$-dimensional data progression $\Delta$. Now comes the heuristic part: let's figure out what happens if we consider only periodic geometric data progressions, rather than all data progressions.

\begin{proposition}
\label{propositioninterlude1}
The following are equivalent:
\begin{itemize}
\item[(A)] $C_{\tmax,\Psi_\alpha}(\Delta) < \infty$ for every periodic geometric $d$-dimensional data progression $\Delta$.
\item[(B)] $\alpha \leq \alpha_d := d(d - 1)^{-(d - 1)/d}$.
\end{itemize}
\end{proposition}
In light of Lemma \ref{lemmaperiodicgeometric}, it suffices to prove the following:

\begin{lemma}
\label{lemmainterlude1}
The unique minimum of the function
\[
f(\gamma) = \gamma + \gamma^{-(d - 1)}
\]
is attained at the value $\gamma_d = (d - 1)^{1/d}$, where it achieves the value $f(\gamma_d) = \alpha_d$.
\end{lemma}
The proof of this lemma is a calculus exercise which is left to the reader.

Note that $\gamma_d > 1$ if and only if $d\geq 3$. If $d = 2$, we still have $\sup_{\gamma > 1}f(\gamma) = \alpha_d$ which is sufficient to deduce Proposition \ref{propositioninterlude1} from Lemma \ref{lemmaperiodicgeometric}.

In the sequel, the following corollary will be useful:

\begin{corollary}
\label{corollarymaximum}
The unique maximum of the function
\[
f_d(\gamma) = (\alpha_d - \gamma)\gamma^{d - 1}
\]
is attained at the value $\gamma_d$, where it acheives the value $f_d(\gamma_d) = 1$.
\end{corollary}
\begin{proof}
We have
\[
\gamma + \gamma^{-(d - 1)} \geq \alpha_d,
\]
with equality if and only if $\gamma = \gamma_d$; rearranging gives the desired result.
\end{proof}

\draftnewpage
\section{The class of recursively integrable functions}
\label{sectionrecint}

In this section we introduce a class of functions to be used in the proof of Theorem \ref{theoremHmax}, the class of \emph{recursively integrable} functions.

\begin{definition}
\label{definitionrecint}
Fix $t_0 \geq 0$, and let $f:\tplus \to \Rplus$ be a continuous function. We say that $f$ is \emph{recursively integrable} if for some $t_1 \geq t_0$ the differential equation
\begin{equation}
\label{recint}
-g'(x) = g^2(x) + f(x)
\end{equation}
has a solution $g: \Tplus\to\Rplus$. The class of recursively integrable functions will be denoted $\RR$. A solution $g$ of \eqref{recint} will be called a \emph{recursive antiderivative} of $f$ (regardless of its domain and range).
\end{definition}

Note that if $f\in\RR$, then $f$ is integrable, since
\[
\int_{t_1}^\infty f(x)\dee x \leq \int_{t_1}^\infty [g^2(x) + f(x)]\dee x = -\int_{t_1}^\infty g'(x)\dee x = g(t_1) - \lim_{t\to\infty} g(t) \leq g(t_1) < \infty.
\]
Like the class of integrable functions, the class $\RR$ is closed under $\leq$:

\begin{lemma}
\label{lemmacomparison}
If $0\leq f_1 \leq f_2$ and if $f_2\in\RR$, then $f_1\in\RR$.
\end{lemma}
\begin{proof}
Let $g_2:\Tplus\to\Rplus$ be a recursive antiderivative of $f_2$. Let $g_1:[t_1,t_2)\to\R$ be a recursive antiderivative of $f_1$ satisfying $g_1(t_1) = g_2(t_1)$. Such a function $g_1$ exists by the fundamental theorem of ordinary differential equations; moreover, $t_2$ may be chosen so that either $t_2 = \infty$ or $\lim_{t\to t_2} g_1(t) = \pm\infty$. It is clear that $g_1 \geq g_2$. In particular $g_1 \geq 0$. On the other hand, $g_1$ is decreasing so $\lim_{t\to t_2} g_1(t)\neq +\infty$. Thus $t_2 = \infty$ and $g_1:\Tplus\to\Rplus$.
\end{proof}

\begin{remark}
\label{remarkcomparison}
Equivalently, Lemma \ref{lemmacomparison} says that if the differential inequality
\[
-g'(x) \geq g^2(x) + f(x)
\]
has a solution $g:\Tplus\to\Rplus$, then $f$ is recursively integrable.
\end{remark}


%


\ignore{
\begin{proposition}
\label{propositiondiffeq}
The recursive integral $\int_0^\infty \rec f(x)\dee x$ converges if and only if there exists a function $g:\Rplus\to\Rplus$ satisfying \eqref{recint}.
\end{proposition}
\begin{proof}
For each $N$, let $g_N:(t_N,N]\to\Rplus$ denote the unique maximal solution of \eqref{recint} satisfying $g_N(N) = 0$. Then $g_N(0) \uparrow \int_0^\infty\rec f(x) \dee x$, unless $t_N > 0$ for some $N$, in which case $\int_0^\infty\rec f(x)\dee x = \infty$.

Suppose that the recursive integral converges. Fix $t\geq \int_0^\infty \rec f(x)\dee x$, and let $g:\Rplus\to\R$ be the unique solution of the differential equation \eqref{recint} satisfying $g(0) = t$. Given $N\in\N$, it is readily seen that $g_N\leq g$, and so in particular we have $g\geq 0$ on $[0,N]$. Since $N$ was arbitrary, we have $g:\Rplus\to\Rplus$.


Conversely, suppose that $g:\Rplus\to\Rplus$ satisfies \eqref{recint}. It is readily seen that $g_N\leq g$ for all $N$, and in particular $\int_0^\infty \rec f(x)\dee x = \lim_{N\to\infty} g_N(0)\leq g(0)$.
\end{proof}

\begin{remark*}
Notice that the function $g$ of Proposition \ref{propositiondiffeq} does not necessarily satisfy $g(0) = \int_0^\infty \rec f(x)\dee x$; indeed, the above proof shows that it may satisfy $g(0) = t$ for any $t\geq \int_0^\infty \rec f(x)\dee x$.
\end{remark*}
}

However, unlike the class of integrable functions, the class $\RR$ is not closed under scalar multiplication, Indeed, we have:

\begin{lemma}
\label{lemmarecint1}
Fix $C > 0$. The function $f(x) = C/x^2$ is recursively integrable if and only if $C\leq 1/4$.
\end{lemma}
\begin{proof}
Suppose that $C \leq 1/4$. Then there exists $c > 0$ such that $C = c - c^2$. The function $g(x) = c/x$ is a recursive antiderivative of $f$, and thus $f\in\RR$.

Conversely, suppose that $C > 1/4$, and by contradiction suppose that $g:\Tplus\to\Rplus$ is a recursive antiderivative of $f$. Letting $h(x) = xg(x)$, we have
\[
\frac{h(x)}{x^2} - \frac{h'(x)}{x} = \frac{h^2(x)}{x^2} + \frac{C}{x^2},
\]
or
\[
-h'(x) = \frac 1x \left[h^2(x) - h(x) + C\right].
\]
But since $C > 1/4$, there exists $\epsilon > 0$ such that $y^2 - y + C \geq \epsilon$ for all $y\in\R$. Thus
\[
-h'(x) \geq \frac{\epsilon}{x}\cdot
\]
It follows that $h(x)\to -\infty$ as $x\to\infty$, contradicting that $g:\tplus\to\Rplus$.
\end{proof}

If $f$ is a function such that the limit $\lim_{x\to\infty} x^2 f(x)$ exists and is not equal to $1/4$, then Lemmas \ref{lemmarecint1} and \ref{lemmacomparison} can be used to determine whether or not $f\in\RR$. This leads to the question: what if $\lim_{x\to\infty} x^2 f(x) = 1/4$? The following lemma provides us with a tool to deal with such functions:

\begin{lemma}
\label{lemmarecint2}
Let $f:\tplus\to\Rplus$. Then $f\in\RR$ if and only if $F\in\RR$, where
\begin{equation*}
\label{1x214}
F(x) := \frac 1{x^2}\left[\frac 14 + f(\log(x))\right].
\end{equation*}
\end{lemma}
\begin{proof}
For any function $g:\Tplus\to\Rplus$, let
\begin{equation}
\label{Gdef}
G(x) := \frac 1x \left[\frac 12 + g(\log(x))\right].
\end{equation}
We have
\begin{align*}
&& -G'(x) &= G^2(x) + F(x)\\
&\Leftrightarrow& xG(x) - x(\dee/\dee x)[xG(x)] &= (xG(x))^2 + x^2 F(x)\\
&\Leftrightarrow& -x(\dee/\dee x)[xG(x)] &= (xG(x) - 1/2)^2 + x^2 F(x) - 1/4\\
&\Leftrightarrow& -g'(\log(x)) &= g^2(\log(x)) + f(\log(x)),
\end{align*}
i.e. $G$ is a recursive antiderivative of $F$ if and only if $g$ is a recursive antiderivative of $f$.

If $g:\Tplus\to\Rplus$ is a recursive antiderivative of $f$, let $G$ be defined by \eqref{Gdef}. Since $G:[e^{t_1},\infty)\to\Rplus$, $F$ is recursively integrable.

Conversely, suppose that $G:\Tplus\to\Rplus$ is a recursive antiderivative of $F$, with $t_1 > 0$. Let $g:[\log(t_1),\infty)\to[-1/2,\infty)$ be defined by \eqref{Gdef}; then $g$ is a recursive antiderivative of $f$. To complete the proof we must show that $g$ is nonnegative. But \eqref{recint} together with the inequality $f\geq 0$ show that
\begin{equation}
\label{gbounds}
-g'(x) \geq g^2(x) \geq 0.
\end{equation}
In particular $g$ is decreasing. Since $g$ is bounded from below, it follows that $\lim_{x\to\infty} g(x)$ exists. Applying \eqref{gbounds} again, we see that this limit must equal $0$. Since $g$ is decreasing, this implies that $g(x)\geq 0$ for all $x$.
\end{proof}

\begin{remark*}
An alternative proof of Lemma \ref{lemmarecint1} may be given by applying Lemma \ref{lemmarecint2} to the class of constant functions.
\end{remark*}

Applying Lemma \ref{lemmarecint2} repeatedly to Lemma \ref{lemmarecint1} yields the following:

\begin{corollary}
\label{corollaryrecint}
For each $N\geq -1$ and $C\geq 0$, the function
\begin{align*}
f_{N,C}(x) &= \frac14\sum_{n = 0}^N \prod_{i = 0}^n \left(\frac{1}{\log^{(i)}(x)}\right)^2 + C\prod_{i = 0}^{N + 1}\left(\frac{1}{\log^{(i)}(x)}\right)^2\\
&= \frac1{x^2}\left[\frac14 + \frac{1}{\log^2(x)}\left[\frac14 + \cdots + \left(\frac{1}{\log^{(N)}(x)}\right)^2\left[\frac14 + C\left(\frac{1}{\log^{(N + 1)}(x)}\right)^2\right]\cdots\right]\right]
\end{align*}
is recursively integrable if and only if $C \leq 1/4$. (If $N = -1$, then the first summation is equal to $0$ by convention.)
\end{corollary}
\begin{remark*}
There is a resemblance between Corollary \ref{corollaryrecint} and the following well-known theorem: For each $N\geq -1$ and $\alpha\geq 0$, the function
\begin{align*}
f(x) &= \left(\prod_{i = 0}^N \frac{1}{\log^{(i)}(x)}\right)\left(\frac{1}{\log^{(N + 1)}(x)}\right)^\alpha\\
&= \frac{1}{x\log(x)\cdots \log^{(N)}(x) \left(\log^{(N + 1)}(x)\right)^\alpha}
\end{align*}
is integrable on an interval of the form $\tplus$ if and only if $\alpha > 1$.
\end{remark*}

We next show that Corollary \ref{corollaryrecint} can be used to determine whether or not $f\in\RR$ whenever $f$ is a Hardy $L$-function.

\begin{proposition}
\label{propositionLfunctions}
If $f$ is a Hardy $L$-function, then there exist $N\in\N$ and $C \geq 0$ such that
\begin{equation}
\label{Lfunctions1}
f(x) \leq f_{N,C}(x) \text{ for all $x$ sufficiently large if $C\leq 1/4$}
\end{equation}
and
\begin{equation}
\label{Lfunctions2}
f(x) \geq f_{N,C}(x) \text{ for all $x$ sufficiently large if $C > 1/4$}.
\end{equation}
We have $f\in\RR$ or $f\notin\RR$ according to whether the former or the latter holds.
\end{proposition}
The second assertion is of course a direct consequence of Corollary \ref{corollaryrecint} and Lemma \ref{lemmacomparison}.
\begin{proof}
Let $N$ be the \emph{order} of $f$ as defined in \cite[\64]{Hardy2}, and consider the function
\[
g(x) = \prod_{i = 0}^N\left(\log^{(i)}(x)\right)^2 \left[4f(x) - \sum_{n = 0}^N \prod_{i = 0}^n \left(\frac{1}{\log^{(i)}(x)}\right)^2\right].
\]
Note that for each $C \geq 0$, we have $f(x) \leq f_{N,C}(x)$ if and only if $g(x) \leq 4C(\log^{(n + 1)}(x))^{-2}$. On the other hand, it is readily seen that $g$ is a Hardy $L$-function of order $\leq N$. So by \cite[Theorem 3]{Hardy2}, there exists $\epsilon > 0$ such that either
\[
g(x) \leq \left(\log^{(N)}(x)\right)^{-\epsilon} \text{ for all $x$ sufficiently large},
\]
or
\[
g(x) \geq \epsilon \text{ for all $x$ sufficiently large}.
\]
In the first case, we have $g(x) \leq (\log^{(n + 1)}(x))^{-2}$ for all $x$ sufficiently large, so \eqref{Lfunctions1} holds with $C = 1/4$. In the second case, we have $g(x) \geq 2(\log^{(n + 1)}(x))^{-2}$ for all $x$ sufficiently large, so \eqref{Lfunctions2} holds with $C = 1/2$.
\end{proof}


One more fact about transformations preserving recursive integrability will turn out to be useful:

\begin{lemma}
\label{lemmarecint3}
Fix $\lambda > 0$. A function $f:\tplus\to\Rplus$ is recursively integrable if and only if the function
\[
f_\lambda(x) = \lambda^2 f(\lambda x)
\]
is recursively integrable.
\end{lemma}
\begin{proof}
If $g$ is a recursive antiderivative of $f$, then $g_\lambda(x) = \lambda g(\lambda x)$ is a recursive antiderivative of $f_\lambda$. Since $f = (f_\lambda)_{1/\lambda}$, the backwards direction follows from the forwards direction.
\end{proof}

We next discuss the robustness of the concept of recursive integrability. As we have seen, it is not preserved under scalar multiplication. In particular, the sum of two recursively integrable functions is not necessarily recursively integrable. However, there are certain functions which can be safely added to a recursively integrable function without affecting its recursive integrability.

In what follows, $\HH$ denotes a Hardy field (cf. Appendix \ref{sectionhardyfields}) which contains the exponential and logarithm functions and is closed under composition. For example, $\HH$ can be (and must contain) the class of Hardy $L$-functions described in the introduction.

\begin{definition}
A nonnegative function $f_2\in\HH$ is \emph{ignorable} if for every function $f_1\in\RR\cap\HH$, we have $f_1 + f_2\in\RR$.
\end{definition}
Note that the sum of any two ignorable functions is ignorable. Moreover, if $f_2$ is ignorable and $0\leq f_1\leq f_2$, then $f_1$ is ignorable (assuming $f_1\in\HH$). By Archimedes' principle, it follows that the class of ignorable functions is closed under (nonnegative) scalar multiplication.

\begin{lemma}
\label{lemmaignorable}
For every $\epsilon > 0$, the function $f_2(x) = 1/x^{2 + \epsilon}$ is ignorable.
\end{lemma}
\begin{proof}
Fix $f_1\in\RR\cap\HH$, and let $g_1:\Tplus\to\Rplus$ be a recursive antiderivative of $f_1$. Fix $C > 1/\epsilon$, and let
\[
g(x) := g_1(x) + \frac{C}{x^{1 + \epsilon}}\cdot
\]
Then
\begin{align*}
-g'(x) - g^2(x)
&= -\left(g_1'(x) - \frac{C(1 + \epsilon)}{x^{2 + \epsilon}}\right) - \left(g_1^2(x) + \frac{2Cg_1(x)}{x^{1 + \epsilon}} + \frac{C^2}{x^{2 + 2\epsilon}}\right)\\
&= f_1(x) + \frac{C}{x^{2 + \epsilon}}\left[1 + \epsilon - 2x g_1(x) - \frac{C}{x^\epsilon}\right].
\end{align*}
Since $f_1\in\HH$, we have either
\begin{equation}
\label{f1leq}
f_1(x) \leq \frac{1}{4x^2} \text{ for all sufficiently large $x$},
\end{equation}
or
\begin{equation}
\label{f1geq}
f_1(x) \geq \frac{1}{4x^2} \text{ for all sufficiently large $x$}
\end{equation}
(Lemma \ref{lemmahardy}). If \eqref{f1leq} holds, then Lemmas \ref{lemmarecint1} and \ref{lemmarecint2} automatically show that $f_1 + f_2\in\RR$. So we suppose that \eqref{f1geq} holds. Let $f_3,g_3$ be defined by the equations
\begin{align*}
f_1(x) &= \frac{1}{x^2}\left[\frac14 + f_3(\log(x))\right]\\
g_1(x) &= \frac{1}{x}\left[\frac12 + g_3(\log(x))\right].
\end{align*}
Then $g_3:[e^{t_1},\infty) \to[-1/2,\infty)$ is a recursive antiderivative of $f_3$. But by \eqref{f1geq}, we have $f_3\geq 0$. By the argument used at the end of the proof of Lemma \ref{lemmarecint2}, the limit $\lim_{x\to\infty}g_3(x)$ exists and is equal to zero. Equivalently, this means that $xg_1(x)\to 1/2$ as $x\to\infty$. Thus
\[
\lim_{x\to\infty} \left[1 + \epsilon - 2x g_1(x) - \frac{C}{x^\epsilon}\right] = \epsilon.
\]
Since $C > 1/\epsilon$, this implies that
\[
-g'(x) - g^2(x) \geq f_1(x) + f_2(x) \text{ for all sufficiently large $x$}.
\]
By Remark \ref{remarkcomparison}, we have $f_1 + f_2\in\RR$.
\end{proof}
\begin{remark*}
Applying Lemma \ref{lemmarecint2} repeatedly shows that for all $N\geq -1$ and $\epsilon > 0$ the function
\begin{align*}
f(x) &= \left(\prod_{i = 0}^N \frac{1}{\log^{(i)}(x)}\right)^2\left(\frac{1}{\log^{(N + 1)}(x)}\right)^{2 + \epsilon}\\
&= \frac{1}{x^2\log^2(x)\cdots \left(\log^{(N)}(x)\right)^2 \left(\log^{(N + 1)}(x)\right)^{2 + \epsilon}}
\end{align*}
is ignorable.
\end{remark*}

We finish this section by providing a number of equivalent conditions to the recursive integrability of a function $f\in\HH$. The following proposition should be thought of as an analogue of the Integral Test which says that a increasing function $f:\Rplus\to\Rplus$ is integrable if and only if the series $\sum_{k = 1}^\infty f(k)$ is summable. It should be noted that as with the Integral Test, the motivation here is not to determine whether a function is recursively integrable by using an equivalent condition, but rather to determine whether one of the equivalent conditions is true by determining whether the function in question is recursively integrable.

\begin{proposition}
\label{propositionintegraltest}
Suppose $f\in\HH$ is nonnegative. Then for any $t\in\R$, the following are equivalent:
\begin{itemize}
\item[(A)] $f\in\RR$.
\item[(B1)] There exists a nonnegative sequence $(S_k)_{k\geq k_0}$ satisfying
\begin{equation}
\label{B1}
S_k - S_{k + 1} \geq S_{k + 1}^2 + f(k).
\end{equation}
\item[(B2)] There exists a nonnegative sequence $(S_k)_{k\geq k_0}$ satisfying
\begin{equation}
\label{B2}
S_k - S_{k + 1} = S_{k + 1}^2 + f(k).
\end{equation}
\item[(C1)] There exists a nonnegative sequence $(S_k)_{k\geq k_0}$ satisfying $S_k\to 0$ and
\begin{equation}
\label{C1}
S_k - S_{k + 1} \geq S_k^2 + t S_k^3 + f(k).
\end{equation}
\item[(C2)] There exists a nonnegative sequence $(S_k)_{k\geq k_0}$ satisfying $S_k\to 0$ and
\begin{equation}
\label{C2}
S_k - S_{k + 1} = S_k^2 + t S_k^3 + f(k).
\end{equation}
\end{itemize}
\end{proposition}
%
%
%
%
%
%
%

\begin{remark*}
Suppose that $f$ satisfies any of the conditions (B1)-(C2). Plugging the formula $S_k\to 0$ into the appropriate equation \eqref{B1} or \eqref{C1} shows that $\limsup_{k\to\infty} f(k) \leq 0$. Since $f\in\HH$ and $f\geq 0$, it follows that $f(x)\to 0$ as $f\to\infty$. Again using the facts that $f\in\HH$ and $f\geq 0$, we deduce that $f$ is decreasing for sufficiently large $x$. Similar reasoning applies if we assume that $f$ satisfies (A).

Thus in the proof of Proposition \ref{propositionintegraltest}, we may assume that $f$ is decreasing on its domain of definition.
\end{remark*}

\begin{remark}
\label{remarkETSf}
Conditions (A), (B1), and (C1) all have the property that when $f_2$ satisfies the condition and $0\leq f_1\leq f_2$, then $f_1$ also satisfies the condition. Thus in proving the equivalences (A) \iff (B1) \iff (C1), it suffices to consider the case where
\begin{equation}
\label{ETSf}
\frac{1}{4x^2} \leq f(x) \leq \frac{1}{x^2} \text{ for all sufficiently large $x$.}
\end{equation}
Indeed, suppose that Proposition \ref{propositionintegraltest} holds whenever $f$ satisfies \eqref{ETSf}. Then by Lemma \ref{lemmarecint1}, the function $f_-(x) = 1/(4x^2)$ satisfies (A), (B1), and (C1) while the function $f_+(x) = 1/x^2$ fails to satisfy them. Now let $f\in\HH$ be arbitrary. If $f$ does not satisfy \eqref{ETSf}, then by Lemma \ref{lemmahardy} either $f(x) \leq f_-(x)$ for all $x$ sufficiently large or $f(x) \geq f_+(x)$ for all $x$ sufficiently large. In the first case, (A), (B1), and (C1) hold while in the second case, (A), (B1), and (C1) fail to hold.
\end{remark}

\begin{proof}[Proof of \text{(A) \implies (B1)}]
If $g:\Tplus\to\Rplus$ is a recursive antiderivative of $f$, then the sequence $S_k = g(k - 1)$ satisfies \eqref{B1}.
\end{proof}
\begin{proof}[Proof of \text{(B1) \implies (B2)}]
Suppose that the sequence $(S_k)_{k\geq k_0}$ satisfies \eqref{B1}. For each $N\geq k_0$, let $(S_k^{(N)})_{k = k_0}^N$ be the unique sequence satisfying \eqref{B2} for $k = k_0,\ldots,N - 1$ and such that $S_N^{(N)} = 0$. Backwards induction shows that for each $k$, the sequence $(S_k^{(N)})_{N\geq k}$ is increasing, and $S_k^{(N)}\leq S_k$ for all $N\geq k$. Let
\[
\w S_k = \lim_{N\to\infty} S_k^{(N)} \in \Rplus.
\]
Then the nonnegative sequence $(\w S_k)_{k\geq k_0}$ satisfies \eqref{B2}.
\end{proof}
\begin{proof}[Proof of \text{(B2) \implies (C1)}]
Suppose that $(S_k)_{k\geq k_0}$ satisfies \eqref{B2}, and that \eqref{ETSf} holds.
\begin{claim}
$k S_k \to 1/2$.
\end{claim}
\begin{subproof}
By \eqref{B2} and \eqref{ETSf}, we have
\begin{equation}
\label{14k2}
S_k - S_{k + 1} \geq S_{k + 1}^2 + \frac{1}{4k^2} \text{ for all $k$ sufficiently large.}
\end{equation}
In analogy with the proof of Lemma \ref{lemmaignorable}, for each $k$ let $T_k\geq -1/2$ satisfy
\[
S_k = \frac1k\left[\frac12 + T_k\right].
\]
Plugging into \eqref{14k2} gives
\[
\frac1k(T_k - T_{k + 1}) + \frac1{k(k + 1)}\left[\frac12 + T_{k + 1}\right] \geq \frac1{(k + 1)^2}\left[T_{k + 1}^2 + T_{k + 1} + \frac14\right] + \frac{1}{4k^2}
\]
and thus
\[
\frac1k(T_k - T_{k + 1}) \geq \frac{T_{k + 1}^2}{(k + 1)^2}\cdot
\]
It follows that the sequence $(T_k)_1^\infty$ is decreasing and bounded from below. Thus the limit $\lim_{k\to\infty}T_k$ exists, and
\[
\infty > \sum_k (T_k - T_{k + 1}) \geq \sum_k \frac{k}{(k + 1)^2} T_{k + 1}^2 \asymp \sum_k \frac1k T_{k + 1}^2,
\]
which implies that $\lim_{k\to\infty}T_k = 0$. Equivalently, $\lim_{k\to\infty} k S_k = 1/2$.
\end{subproof}

In particular, $S_k\to 0$. Fix $C > 0$, and let
\[
\w S_k = S_k + \frac{C}{k^2}\cdot
\]
Then $\w S_k\to 0$ as well. So to complete the proof, we need to show that \eqref{C1} holds for the sequence $(\w S_k)_{k\geq k_0}$. We have
\begin{align*}
\w S_k - \w S_{k + 1}
&\geq S_k - S_{k + 1} + \frac{2C}{(k + 1)^3}\\
&= S_{k + 1}^2 + f(k) + \frac{2C}{(k + 1)^3}\\
&= \w S_k^2 + f(k) + \frac{2C}{(k + 1)^3} - (\w S_k + S_{k + 1})(\w S_k - S_{k + 1})\\
&\geq \w S_k^2 + f(k) + \frac{2C}{(k + 1)^3} - 2\left(S_k + \frac{C}{k^2}\right)\left(S_{k + 1}^2 + f(k) + \frac{C}{k^2}\right).
\end{align*}
Let
\[
E_k = \frac{2C}{(k + 1)^3} - 2\left(S_k + \frac{C}{k^2}\right)\left(S_{k + 1}^2 + f(k) + \frac{C}{k^2}\right) - t \w S_k^3,
\]
so that
\[
\w S_k - \w S_{k + 1} \geq \w S_k^2 + f(k) + t \w S_k^3 + E_k.
\]
So to complete the proof, it suffices to show that if $C$ is large enough, then $E_k \geq 0$ for all $k$ sufficiently large. And indeed,
\begin{align*}
\liminf_{k\to\infty} k^3 E_k
&= 2C - 2\limsup_{k\to\infty}\left[\left(k S_k + \frac{C}{k}\right)\left(k^2 S_{k + 1}^2 + k^2 f(k) + C\right)\right] - t \left(\limsup_{k\to\infty} k \w S_k\right)^3\\
&\geq 2C - 2 (1/2)(1/4 + 1 + C) - t/8 = C - t/8 - 5/4.
\end{align*}
Thus by choosing $C > t/8 + 5/4$, we complete the proof.
\end{proof}
\begin{proof}[Proof of \text{(C1) \implies (C2)}]
Suppose that the sequence $(S_k)_{k\geq k_0}$ satisfies $S_k\to 0$ and \eqref{C1}. Fix $k_1\geq k_0$ large enough so that
\[
S_{k_1} \leq \frac{1}{\max(5,|t| + 1)}\cdot
\]
Then for all $0 < x\leq S_{k_1}$, we have $x^2 + tx^3 > 0$ and $(\dee/\dee x)[x - x^2 - tx^3] \geq 0$. Let $(\w S_k)_{k\geq k_1}$ be the unique sequence satisfying \eqref{C2} and $\w S_{k_1} = S_{k_1}$. An induction argument shows that for all $k\geq k_1$, $S_k \leq \w S_k \leq S_{k_1}$ and $\w S_{k + 1} \leq \w S_k$. In particular the sequence $(\w S_k)_{k\geq k_0}$ is nonnegative. To complete the proof we need to show that $\w S_k\to 0$. Since $(\w S_k)_k$ is decreasing, the limit $L = \lim_{k\to\infty}\w S_k$ exists. Taking the limit of \eqref{C2} we find that
\[
L - L = L^2 + tL^3.
\]
Since $0\leq L \leq S_{k_1}$, this implies that $L = 0$.
\end{proof}
\begin{proof}[Proof of \text{(C1) \implies (A)}]
First suppose $t = 0$. If $(S_k)_{k\geq k_0}$ satisfies $S_k\to 0$ and \eqref{C1}, then let $g$ be the linear interpolation of $(S_k)_{k\geq k_0}$, i.e.
\[
g(x) = S_k + (x - k)(S_{k + 1} - S_k) \text{ for } k \leq x \leq k + 1.
\]
Then
\[
-g'(x) = S_{k + 1} - S_k \geq S_k^2 + f(k) \geq g^2(x) + f(x) \all k < x < k + 1,
\]
so by Remark \ref{remarkcomparison} $f\in\RR$.

Now suppose $t\neq 0$ and that $(S_k)_{k\geq k_0}$ satisfies $S_k\to 0$ and \eqref{C1}. Let $\ell$ be large enough so that
\[
\frac{2}{k} - \left(\frac{2}{k}\right)^2 - t\left(\frac{2}{k}\right)^3 \leq \frac{2}{k + 1} \all k\geq \ell,
\]
and let $k_1\geq k_0$ be large enough so that $S_{k_1 + \ell} \leq 2/\ell$. Then an induction argument shows that
\begin{equation}
\label{Sklowerbound}
S_k \leq \frac{2}{k - k_1} \all k\geq k_1 + \ell.
\end{equation}
In particular, there exists $C > 0$ such that $S_k \leq C/k$ for all $k\geq k_0$. Then
\[
S_k - S_{k + 1} \geq S_k^2 + f(k) - \frac{|t|C^3}{k^3},
\]
and so by the $t = 0$ case of (C1) \implies (A), the function $x\mapsto f(x) - |t|C^3/x^3$ is recursively integrable.  Since the function $x\mapsto |t| C^3/x^3$ is ignorable (Lemma \ref{lemmaignorable}), $f$ is also recursively integrable.
\end{proof}

\draftnewpage
\section{Proof of Theorems \ref{theoremnooptimaldirichlet} and \ref{theoremHmax} and formula \eqref{exponents1}}
\label{sectionHmax}

As in Section \ref{sectionrecint}, $\HH$ denotes a Hardy field which contains the exponential and logarithm functions and is closed under composition, for example the field of Hardy $L$-functions. As in Section \ref{sectioninterlude}, we write
\begin{align*}
\gamma_d &= (d - 1)^{1/d} > 1 \hspace{.3 in}(\text{if }d\geq 3)\\
\alpha_d &= \gamma_d + \gamma_d^{-(d - 1)} = d(d - 1)^{-(d - 1)/d}.
\end{align*}

Theorems \ref{theoremnooptimaldirichlet} and \ref{theoremHmax} and formula \eqref{exponents1} will all follow from the following result:

\begin{proposition}
\label{propositionHmax}
Suppose that $d\geq 3$, and fix $\psi\in\HH$. Then the following are equivalent:
\begin{itemize}
\item[(A)] $\psi$ is Dirichlet on $\R^d$ with respect to $\Hmax$.
\item[(B)] $\psi$ is uniformly Dirichlet on $\R^d$ with respect to $\Hmax$.
\item[(C)] $C_{\Hmax,\psi}(\xx) = 0$ for all $\xx\in\R^d$, i.e. $\psi$ is non-optimally Dirichlet on $\R^d$ with respect to $\Hmax$.
\item[(D)] The function
\[
f_\psi(x) = \frac{2}{d\gamma_d}\left[\alpha_d + \frac{\log\psi(e^{\gamma_d^x})}{\gamma_d^x}\right]
\]
is nonnegative for large values of $x$ and satisfies $f_\psi\notin\RR$.
\end{itemize}
In particular, no function $\psi\in\HH$ is optimally Dirichlet on $\R^d$ with respect to the height function $\Hmax$.
\end{proposition}
\begin{proof}[Proof of Theorems \ref{theoremnooptimaldirichlet} and \ref{theoremHmax} and formula \eqref{exponents1} assuming Proposition \ref{propositionHmax}]
Suppose that Proposition \ref{propositionHmax} is true. Then for all $\alpha\geq 0$, $\psi_\alpha$ is Dirichlet on $\R^d$ with respect to $\Hmax$ if and only if $\alpha < \alpha_d$. It follows that $\omega_d(\Hmax) = \alpha_d$, demonstrating the formula \eqref{exponents1}.

Since Theorem \ref{theoremnooptimaldirichlet} is a restatement of the equivalence of (A) and (C) of Proposition \ref{propositionHmax}, to complete the proof it suffices to prove Theorem \ref{theoremHmax}. Specifically, given $N\geq 1$ and $C\geq 0$, we must show that the function $\psi_{N,C}$ of Theorem \ref{theoremHmax} satisfies the equivalent conditions (A)-(D) of Proposition \ref{propositionHmax} if and only if $C > 1$. Considering condition (D), we must show that $f_{\psi_{N,C}}\in\RR$ if and only if $C\leq 1$. But
\[
f_{\psi_{N,C}}(x) = \log^2(\gamma_d) f_{N - 2,C/4}(x\log(\gamma_d)),
\]
so this follows from Corollary \ref{corollaryrecint} and Lemma \ref{lemmarecint3}.
\end{proof}


The proof of Proposition \ref{propositionHmax} will be divided into three parts: the proof of (D) \implies (B), which constitutes the hardest part of the argument; the proof of (C) \implies (D), which is essentially the proof of (D) \implies (B) in reverse, but made easier due to the explicitness of the data structure in question; and finally, the reduction of the theorem to those two implications, which is essentially a corollary of Lemma \ref{lemmaignorable}.


\begin{remark*}
Throughout the proof we will assume that
\begin{equation}
\label{assumptionphi}
\frac{1}{4x^2} \leq f_\psi(x) \leq \frac{1}{x^2} \text{ for all $x$ sufficiently large.}
\end{equation}
The justification of this assumption follows along the same lines as Remark \ref{remarkETSf}. Specifically, suppose that Proposition \ref{propositionHmax} holds whenever $\psi$ satisfies \eqref{assumptionphi}. Let $\psi_-$ and $\psi_+$ denote the functions for which equality holds in the left and right hand inequalities of \eqref{assumptionphi}, respectively. Then by Lemma \ref{lemmarecint1}, $\psi_+$ satisfies (A)-(D) of Proposition \ref{propositionHmax} while $\psi_-$ fails to satisfy them. Now let $\psi\in\HH$ be arbitrary. If $\psi$ does not satisfy \eqref{assumptionphi}, then by Lemma \ref{lemmahardy} either $\psi(q) \geq \psi_+(q)$ for all $q$ sufficiently large or $\psi(q) \leq \psi_-(q)$ for all $q$ sufficiently large. In the first case, we have $C_{\Hmax,\psi}\leq C_{\Hmax,\psi_+}$ and so (A)-(D) of Proposition \ref{propositionHmax} hold. In the second case, we have $C_{\Hmax,\psi} \geq C_{\Hmax,\psi_-}$ and so (A)-(D) of Proposition \ref{propositionHmax} fail to hold.
\end{remark*}

\begin{remark}
\label{remarkreverseimplication}
When reading the proof of (D) \implies (B), one should check that the implications \eqref{contradictionhypothesis5} \implies \eqref{contradictionhypothesis6} \implies \eqref{contradictionhypothesis7} \implies \eqref{contradictionhypothesis8} are all invertible if one assumes the following facts about $\Delta$: $\max(\Delta_k) = A_k$ for all $k\in\N$, and $\Delta$ is eventually periodic in the sense of Claim \ref{claimrotating}. The converse directions will be used in the proof of (C) \implies (D).
\end{remark}

\begin{notation*}
The following notations will be used in the course of the proof:
\begin{align*}
\Psi(b) &= -\log\psi(e^b)\\
\Phi(b) &= \alpha_d - \frac{\Psi(b)}{b}\cdot
\end{align*}
Note that according to these notations,
\begin{equation}
\label{notation}
f_\psi(x) = \frac{2}{d\gamma_d} \Phi(\gamma_d^x).
\end{equation}
\end{notation*}

\subsection{Proof of \text{(D) \implies (B)}}
We prove the contrapositive. Suppose that $\sup_{\R^d\butnot\Q^d} C_{\Hmax,\psi} = \infty$, and we will show that $f_\psi\in\RR$. By Lemma \ref{lemmacondition}, we have $\sup_\Delta C_{\tmax,\Psi}(\Delta) = \infty$, where the supremum is taken over $d$-dimensional data progressions $\Delta$. In particular, there exists a $d$-dimensional data progression $\Delta = (A_k,i_k)_{k = 1}^\infty$ such that $C_{\tmax,\Psi}(\Delta) > 0$. It follows that
\begin{equation}
\label{contradictionhypothesis5}
b_k^{(i_k)} + b_{k + 1}^{(i_k)} \leq \Psi(\max(\Delta_k))
\end{equation}
for all $k$ sufficiently large.

\begin{claim}
We may suppose without loss of generality that $\max(\Delta_k) = A_k$ for all $k\in\N$.
\end{claim}
\begin{proof}
Consider the set $K = \{k\in\N : \max(\Delta_{k + 1}) > \max(\Delta_k)\}$. The set $K$ is infinite by part (II) of the definition of a data progression. Let $(k_\ell)_1^\infty$ be the unique increasing indexing of $K$, and consider the data progression $\w\Delta = (\max(\Delta_{k_\ell}),i_{k_\ell})_{\ell = 1}^\infty$. Note that for all $\ell\in\N$ and $i = 1,\ldots,d$,
\begin{align*}
\w b_\ell^{(i)} &\leq b_{k_\ell}^{(i)}\\
\max(\w\Delta_\ell) &= \w A_\ell = \max(\Delta_{k_\ell}).
\end{align*}
Moreover, if $k = k_\ell$, then
\[
b_{k + 1}^{(i_k)} = \max(\Delta_{k + 1}) = \max(\Delta_{k_{\ell + 1}}) = A_{\ell + 1} = \w b_{\ell + 1}^{(\w i_\ell)}.
\]
Plugging all these into \eqref{contradictionhypothesis5} gives
\[
\w b_\ell^{(\w i_\ell)} + \w b_{\ell + 1}^{(\w i_\ell)} \leq \Psi(\max(\w\Delta_\ell)),
\]
i.e. \eqref{contradictionhypothesis5} holds for the data progression $\w\Delta$.
\end{proof}

So in what follows, we assume that $\max(\Delta_k) = A_k$ for all $k\in\N$. Using this fact together with \eqref{bkirec}, \eqref{contradictionhypothesis5} becomes
\[
b_k^{(i_k)} \leq \Psi(A_k) - A_{k + 1}.
\]
Letting $t_k = A_{k + 1}/A_k$, we may rewrite the above equation as
\begin{equation}
\label{contradictionhypothesis6}
b_k^{(i_k)} \leq A_k(\alpha_d - \Phi(A_k) - t_k).
\end{equation}
For each $k\in\N$ let
\[
f_k = \frac{\prod_{i = 1}^d b_k^{(i)}}{(A_k)^d};
\]
using \eqref{bkirec}, \eqref{contradictionhypothesis6} then becomes
\begin{equation}
\label{contradictionhypothesis7}
\frac{f_k}{f_{k + 1}} \leq (\alpha_d - \Phi(A_k) - t_k) t_k^{d - 1}.
\end{equation}

\begin{claim}
\label{claimincreasing}
For some $k_1\in\N$, the sequence $(f_k)_{k_1}^\infty$ is increasing.
\end{claim}
\begin{proof}
By \eqref{assumptionphi}, we have $\Phi(b) \geq 0$ for all $b$ sufficiently large. Thus by Corollary \ref{corollarymaximum},
\begin{equation}
\label{fincreasing}
\frac{f_k}{f_{k + 1}} \leq (\alpha_d - t_k) t_k^{d - 1} \leq 1
\end{equation}
for all $k$ sufficiently large.
\end{proof}

\begin{claim}
\label{claimtkgammad}
$t_k \to \gamma_d$ as $k\to\infty$.
\end{claim}
\begin{proof}
We clearly have $f_k\leq 1$ for all $k$, so by Claim \ref{claimincreasing}, the sequence $(f_k)_1^\infty$ converges to a positive number. Thus $\frac{f_k}{f_{k + 1}} \to 1$. Combining with \eqref{fincreasing}, we see that $(\alpha_d - t_k) t_k^{d - 1} \to 1$. Applying Corollary \ref{corollarymaximum} again, we get $t_k \to \gamma_d$.
\end{proof}

\begin{claim}
\label{claimrotating}
$\Delta$ is eventually periodic in the following sense: there exists a permutation $\sigma:\{1,\ldots,d\}\to\{1,\ldots,d\}$ such that for all $k$ sufficiently large,
\begin{equation}
\label{rotating}
i_k = \sigma(j_k) \text{ where }j_k = k \text{ (mod $d$)}.
\end{equation}
\end{claim}
\begin{proof}
Combining \eqref{contradictionhypothesis6} and Claim \ref{claimtkgammad}, we see that
\[
\limsup_{k\to\infty} \frac{b_k^{(i_k)}}{A_k} \leq \alpha_d - \gamma_d = \gamma_d^{-(d - 1)}.
\]
On the other hand, for each $j = 0,\ldots,d - 2$, by Claim \ref{claimtkgammad} we have
\[
\lim_{k\to\infty} \frac{A_{k - j}}{A_k} = \gamma_d^{-j} > \gamma_d^{-(d - 1)}.
\]
It follows that $b_k^{(i_k)} = A_{\ell(i_k,k) + 1} \neq A_{k - j}$ for all $k$ sufficiently large. In particular $\ell(i_k,k) \neq k - j - 1$. Now fix $k_2\in\N$ such that for all $k\geq k_2$ and $j = 0,\ldots,d - 2$, we have $\ell(i_k,k) \neq k - j - 1$. Then $\ell(i_k,k) \leq k - d$, so $i_{k - j} \neq i_k$ for all $j = 1,\ldots,d - 1$. In particular, the sets
\[
\{i_k,\ldots,i_{k + d - 1}\} \text{ and } \{i_{k + 1},\ldots,i_{k + d}\}
\]
both contain $d$ distinct elements. It follows that $i_k = i_{k + d}$, so the sequence $(i_k)_{k\geq k_2}$ is periodic of period $d$. At this point, it is clear that \eqref{rotating} holds for some permutation $\sigma$.
\end{proof}

\begin{corollary}
\label{corollaryfk}
For all sufficiently large $k$,
\begin{equation}
\label{fk}
f_k = \prod_{j = 1}^{d - 1} \frac{A_{k - j}}{A_k} = \prod_{j = 1}^{d - 1}\frac{1}{t_{k - j}^{d - j}}\cdot
\end{equation}
\end{corollary}
\begin{proof}
Fix $k$ large enough such that the set $\{i_{k - d},\ldots,i_{k - 1}\}$ contains $d$ distinct elements; this is possible by Claim \ref{claimrotating}. It follows that
\[
\{\ell(i,k) : i = 1,\ldots,d\} = \{k - 1,\ldots,k - d\}
\]
and thus
\[
\prod_{i = 1}^d b_k^{(i)} = \prod_{i = 1}^d A_{\ell(i,k) + 1} = \prod_{j = 1}^d A_{k - j + 1}.
\]
Dividing both sides by $(A_k)^d$ finishes the proof.
\end{proof}

\begin{corollary}
\label{corollaryAktk}
For all $k$,
\begin{equation}
\label{Aktk}
A_k \gtrsim \gamma_d^k.
\end{equation}
\end{corollary}
\begin{proof}
By Claim \ref{claimtkgammad},
\[
f_k \tendsto k \prod_{j = 1}^{d - 1} \frac{1}{\gamma_d^{d - j}} = \gamma_d^{-\binom d2}.
\]
By Claim \ref{claimincreasing}, it follows that $f_k \leq \gamma_d^{-\binom d2}$ for all $k$ sufficiently large. Let $k_3$ be large enough so that \eqref{fk} holds for all $k\geq k_3$; then
\[
\gamma_d^{-k\binom d2} \gtrsim \prod_{\ell = k_3}^{k - 1} f_\ell = \prod_{j = 1}^{d - 1}\prod_{\ell = k_3}^{k - 1}\frac{1}{t_{\ell - j}^{d - j}} = \prod_{j = 1}^{d - 1}\left(\frac{A_{k_3 - j}}{A_{k - j}}\right)^{d - j},
\]
and thus
\[
A_k^{\binom d2} \geq \prod_{j = 1}^{d - 1} (A_{k - j})^{d - j} \gtrsim \gamma_d^{k\binom d2}.
\]
Taking $\binom d2$th roots completes the proof.
\end{proof}

Using Corollary \ref{corollaryfk}, \eqref{contradictionhypothesis7} becomes
\[
\prod_{j = 1}^{d - 1}\frac{t_k}{t_{k - j}} \leq (\alpha_d - \Phi(A_k) - t_k)t_k^{d - 1},
\]
or equivalently
\[
t_k \leq \alpha_d - \Phi(A_k) - \prod_{j = 1}^{d - 1}\frac{1}{t_{k - j}}\cdot
\]
Writing $s_k = t_k/\gamma_d - 1$, a few arithmetic calculations show that the above inequality is equivalent to
\begin{equation}
\label{contradictionhypothesis8}
s_k \leq \frac{1}{d - 1}\left[1 - \prod_{j = 1}^{d - 1}\frac{1}{1 + s_{k - j}}\right] - \frac{\Phi(A_k)}{\gamma_d}\cdot
\end{equation}
\ignore{
Letting $m$ be large enough so that $\gamma_d^m$ is greater than the implied constant of \eqref{Aktk}, we have
\begin{equation}
\label{contradictionhypothesis9}
s_k \leq \frac{1}{d - 1}\left[1 - \prod_{j = 1}^{d - 1}\frac{1}{1 + s_{k - j}}\right] - \frac{\Phi(\gamma_d^{k - m})}{\gamma_d}
\end{equation}
for all $k$ sufficiently large.}
Consequently, it becomes important to study behavior the function
\[
f(x_1,\ldots,x_{d - 1}) = 1 - \prod_{j = 1}^{d - 1}\frac{1}{1 + x_j}
\]
near the origin. We calculate the gradient and Hessian of $f$ at $\0$:
\begin{align*}
f'(\0) &= \sum_{j = 1}^{d - 1} \ee_j\\
f''(\0) &= -\left[\sum_{j = 1}^{d - 1} \ee_j^2 + \left(\sum_{j = 1}^{d - 1}\ee_j\right)^2\right]
\end{align*}
Since $f(\0) = 0$, this means that $f$ can be estimated in a neighborhood of the origin by the formula
\begin{equation}
\label{fapprox}
f(\xx) = \sum_{j = 1}^{d - 1} x_j - \frac12\left[\sum_{j = 1}^{d - 1} x_j^2 + \left(\sum_{j = 1}^{d - 1}x_j\right)^2\right] + O(\|\xx\|^3).
\end{equation}
In fact, we can be explicit: \eqref{fapprox} holds whenever $\|\xx\| \leq 1/2$.

Continuing with the proof, for $k\in\N$ let
\[
\phi_k = \frac{2}{d\gamma_d}\Phi(A_k)
\]
(cf. \eqref{notation}).

\begin{claim}
\label{claimphiprime}
For all $k$ sufficiently large,
\[
|\phi_{k + 1} - \phi_k| \lesssim \frac{1}{k^3}\cdot
\]
\end{claim}
\begin{proof}
Since $f_\psi\in\HH$, we may differentiate the inequalities \eqref{assumptionphi} (cf. Lemma \ref{lemmahardy2}) to get
\begin{equation}
\label{phiprime}
|f_\psi'(x)| \leq \left|\frac{\dee}{\dee x}\left[\frac{1}{x^2}\right]\right| = \frac{2}{x^3} \text{ for all $x$ sufficiently large.}
\end{equation}
Using \eqref{notation} and applying the fundamental theorem of calculus, we have
\begin{align*}
|\phi_{k + 1} - \phi_k|
&= |f_\psi(\log_{\gamma_d}(A_{k + 1})) - f_\psi(\log_{\gamma_d}(A_k))| \noreason\\
&\leq \frac{2}{\log_{\gamma_d}^3(A_k)} \log_{\gamma_d}(t_k) \noreason\\
&\lesssim \frac{1}{k^3}\cdot \by{Claim \ref{claimtkgammad} and Corollary \ref{corollaryAktk}}
&\qedhere\end{align*}
\end{proof}

Now fix $C_1 > 0$ large to be determined, then fix $\delta > 0$ small to be determined (possibly depending on $C_1$), and finally fix $k_0\in\N$ large to be determined (possibly depending on both $\delta$ and $C_1$). Let $(S_k)_{k = k_0}^\infty$ be the unique sequence defined by the equations
\begin{equation}
\label{Skdef}
S_{k + 1} = S_k - S_k^2 - \phi_k + \frac{C_1}{k^3} + C_1 |S_k|^3, \;\;\; S_{k_0} = \delta.
\end{equation}

The following claim is the heart of the proof:

\begin{claim}
\label{claimtkbound}
If $k_0$ and $C_1$ are sufficiently large and $\delta$ is sufficiently small (with $k_0$ allowed to depend on $\delta$, which is in turn allowed to depend on $C_1$), then
\begin{equation}
\label{tkbound}
-\frac 1{\max(2,C_1)} \leq s_k \leq S_k \leq \delta \leq \frac 1{\max(2,C_1)}
\end{equation}
for all $k\geq k_0$.
\end{claim}
\begin{proof}
Throughout the proof, we will assume that $\delta < 1/\max(2,C_1)$ and that $k_0\geq 4C_1$. Since $\delta$ and $k_0$ are both allowed to depend on $C_1$, these assumptions are justified. In particular, the rightmost inequality of \eqref{tkbound} requires no proof.

By Claim \ref{claimtkgammad}, we have $s_k\to 0$. Thus, the leftmost inequality of \eqref{tkbound} can be achieved simply by an appropriate choice of $k_0$.

The proof of the two middle inequalities of \eqref{tkbound} is by strong induction on $k$. \\

{\bf Base Case:} $k = k_0,\ldots,k_0 + d - 2$. For this part of the proof, we'll think of $C_1,\delta > 0$ as being fixed. Define the sequence $(T_j)_{j = 0}^{d - 2}$ via the formula
\[
T_{j + 1} = T_j - T_j^2 + C_1 |T_j|^3, \;\; T_0 = \delta.
\]
Since $\delta < 1/\max(2,C_1)$, the sequence $(T_j)_{j = 0}^{d - 2}$ is strictly decreasing and strictly positive. Note that for each $j = 0,\ldots,d - 2$,
\[
S_{k_0 + j}^{(k_0)} \tendsto{k_0} T_j,
\]
where the superscript of $k_0$ is merely making explicit the fact that the sequence $(S_k)_{k\geq k_0}$ depends on $k_0$. On the other hand,
\[
s_{k_0 + j} \tendsto{k_0} 0 < T_j.
\]
So if $k_0$ is sufficiently large, then \eqref{tkbound} holds for $k = k_0 + j$. \\

{\bf Inductive Step:} Fix $\ell\geq k_0 + d - 1$, and suppose that that \eqref{tkbound} holds for $k = \ell - d + 1,\ldots,\ell - 1$. We claim that \eqref{tkbound} holds for $k = \ell$.

\begin{subclaim}
\label{subclaimdecreasing}
For $j = 1,\ldots,d - 1$,
\[
S_{\ell - j + 1} \leq S_{\ell - j}.
\]
\end{subclaim}
\begin{subproof}
By \eqref{assumptionphi}, we have
\[
\phi_k \geq \frac1{4k^2}\cdot
\]
Since $k_0\geq 4C_1$, combining with \eqref{Skdef} gives
\begin{equation}
\label{strongdecreasing}
S_{k + 1} \leq S_k - S_k^2 + C_1 |S_k|^3 \all k\geq k_0.
\end{equation}
Plugging in $k = \ell - j$, we have $|S_k| \leq 1/C_1$ by the induction hypothesis. Thus $S_{k + 1} \leq S_k$.
\end{subproof}

In particular, plugging in $j = 1$ and using the induction hypothesis, we see that the third inequality of \eqref{tkbound} holds for $k = \ell$. So to complete the proof, we need only to demonstrate that the second inequality of \eqref{tkbound} holds for $k = \ell$.

\begin{subclaim}
\label{subclaimignorable1}
For $j = 1,\ldots,d - 1$,
\begin{align*}
|S_{\ell - j}| &\lesssim \max(1/\ell^2,|S_{\ell - j + 1}|)\\
|S_{\ell - j + 1}| &\lesssim \max(1/\ell^2,|S_{\ell - j}|)\\
\end{align*}
\end{subclaim}
\begin{remark*}
We emphasize that here and below, the implied constants of asymptotics may not depend on $C_1$, $\delta$, or $k_0$.
\end{remark*}
\begin{subproof}
By \eqref{assumptionphi}, we have
\[
\phi_k \leq \frac1{k^2}\cdot
\]
On the other hand, since $k_0\geq C_1$ we have $C_1/k^3 \leq 1/k^2$ for all $k\geq k_0$. Letting $k = \ell - j$, combining with \eqref{Skdef}, and writing $x = S_{\ell - j}$, $y = S_{\ell - j + 1}$, we have
\[
\left|x - x^2 + C_1|x|^3 - y\right| \lesssim \frac{1}{(\ell - j)^2} \asymp \frac{1}{\ell^2}\cdot
\]
By the induction hypothesis, we have
\begin{equation}
\label{xbounds}
|x|\leq 1/\max(2,C_1).
\end{equation}
It follows that
\[
|y| \lesssim \max(1/\ell^2, |x - x^2 + C_1|x|^3|) \lesssim \max(1/\ell^2,|x|).
\]
On the other hand, \eqref{xbounds} also implies that $x - x^2 + C_1|x|^3 \leq x$. In particular, if $x$ is negative then
\[
|x| \leq \left|x - x^2 + C_1|x|^3\right| \lesssim \max(1/\ell^2,|y|).
\]
Finally, if $x$ is positive, then we have
\[
|x| = x \asymp x - x^2 \leq x - x^2 + C_1|x|^3 \lesssim \max(1/\ell^2,|y|).
\]
\end{subproof}

\begin{subclaim}
\label{subclaimignorable2}
Let
\[
a_\ell = \max\left(\frac{1}{\ell},|S_\ell|\right).
\]
Then $a_\ell \lesssim 1/C_1$.
\end{subclaim}
\begin{subproof}
Since $\ell \geq k_0 \geq C_1$, we have $1/\ell \leq 1/C_1$. On the other hand, by Subclaim \ref{subclaimignorable1} and the induction hypothesis we have
\[
|S_\ell| \lesssim \max\left(\frac{1}{\ell^2},|S_{\ell - 1}|\right) \leq \frac{1}{C_1}\cdot
\]
\end{subproof}

\begin{definition}
\label{definitionignorable}
For the purposes of this proof, an expression will be called \emph{negligible} if its absolute value is less than a constant times $a_\ell^3$. (The constant must be independent of $C_1$, $\delta$, and $k_0$.) We'll write $A\sim B$ if the difference between two expressions $A$ and $B$ is negligible.
\end{definition}

Note that by Subclaim \ref{subclaimignorable1}, we have $|S_{\ell - j}| \lesssim a_\ell$ for all $j = 0,\ldots,d - 1$. It follows from this and \eqref{Skdef} (keeping in mind Subclaim \ref{subclaimignorable2} and Claim \ref{claimphiprime}) that $|S_{\ell - j + 1} - S_{\ell - j}| \lesssim a_\ell^2$, and thus that
\[
S_{\ell - j_1} (S_{\ell - j_2} - S_{\ell - j_2 + 1}) \sim 0
\]
for all $j_1 = 0,\ldots,d - 1$ and $j_2 = 1,\ldots,d - 1$. It follows that
\[
S_{\ell - j_1} S_{\ell - j_2} \sim S_\ell^2
\]
for all $j_1,j_2 = 0,\ldots,d - 1$.

We are now ready to continue our calculation:

\label{calculations1}
\begin{align*}
s_\ell
&\leq \frac{1}{d - 1}f(S_{\ell - d + 1},\ldots,S_{\ell - 1}) - \frac{d}{2} \phi_\ell \by{\eqref{contradictionhypothesis8}}\\
&\sim \frac{1}{d - 1}\left[\sum_{j = 1}^{d - 1} S_{\ell - j} - \frac12 \left[\sum_{j = 1}^{d - 1} S_\ell^2 + \left(\sum_{j = 1}^{d - 1} S_\ell\right)^2\right]\right] - \frac{d}{2} \phi_\ell \by{\eqref{fapprox}}\\
&= \frac{1}{d - 1}\left[\sum_{j = 1}^{d - 1} S_{\ell - j} - \binom d2 S_\ell^2\right] - \frac{d}{2} \phi_\ell\\
\end{align*}
\begin{align*}
\sum_{j = 1}^{d - 1} [S_{\ell - j} - S_\ell]
&= \sum_{j = 1}^{d - 1}\sum_{i = 1}^j \left[S_{\ell - i}^2 +  \phi_{\ell - i} - C_1 \left[\frac{1}{(\ell - i)^3} + S_{\ell - i}^3\right]\right] \by{\eqref{Skdef}}\\
&\sim \sum_{j = 1}^{d - 1}\sum_{i = 1}^j \left[S_\ell^2 +  \phi_\ell - C_1 \left[\frac{1}{\ell^3} + |S_\ell|^3\right]\right]\\
&= \binom d2 \left[S_\ell^2 +  \phi_\ell - C_1 \left[\frac{1}{\ell^3} + |S_\ell|^3\right]\right]\\
s_\ell - S_\ell
&\leq \frac{1}{d - 1}f(S_{\ell - d + 1},\ldots,S_{\ell - 1}) - \frac{d}{2} \phi_\ell - S_\ell\\
&\sim \frac d2\left[ \phi_\ell - C_1 \left[\frac{1}{\ell^3} + |S_\ell|^3\right]\right] - \frac{d}{2} \phi_\ell\\
&= -\frac d2 C_1 \left[\frac{1}{\ell^3} + |S_\ell|^3\right] \leq - \frac d2 C_1 a_\ell^3
\end{align*}
\label{calculations2}

By the definition of negligibility, we have
\[
s_\ell - S_\ell \leq C_2 a_\ell^3 - \frac d2 C_1 a_\ell^3
\]
for some constant $C_2$ independent of $C_1$, $\delta$, and $k_0$. By letting $C_1 = (2/d)C_2$, we have $s_\ell \leq S_\ell$, completing the proof.
\end{proof}

Having finished the proof of Claim \ref{claimtkbound}, we continue with the proof of Proposition \ref{propositionHmax} (D) \implies (B). Since $S_k\geq s_k\to 0$ and since the sequence $(S_k)_{k\geq k_0}$ is decreasing by Subclaim \ref{subclaimdecreasing}, we have $S_k \geq 0$ for all $k\geq k_0$. The proof of Proposition \ref{propositionintegraltest} (C1) \implies (A) now shows that there exists $C_3 > 0$ such that $S_k \leq C_3/k$ for all $k\geq k_0$ (cf. \eqref{Sklowerbound}). Combining with \eqref{tkbound}, we see that

\begin{align*}
A_k
= A_{k_0}\prod_{\ell = k_0}^{k - 1} \gamma_d(1 + s_\ell)
\leq A_{k_0}\gamma_d^{k - k_0} \prod_{\ell = k_0}^{k - 1} (1 + C_3/\ell)
= A_{k_0} \gamma_d^{k - k_0}\prod_{\ell = k_0}^{k - 1} \frac{\ell + C_3}{\ell}
\leq C_4 \gamma_d^k k^n,
\end{align*}
where $n = \lceil C_3\rceil$ and $C_4 > 0$. So for all sufficiently large $k$,
\[
\phi_k \geq \frac{2}{d\gamma_d} \Phi(C_4 \gamma_d^k k^n).
\]
Applying the fundamental theorem of calculus to \eqref{phiprime} gives
\begin{align*}
f_\psi(k) - \phi_k
&\leq f_\psi(k) -  \frac{2}{d\gamma_d}\Phi(C_4 \gamma_d^k k^n)\\
&= f_\psi(k) - f_\psi\left(k + \log_{\gamma_d}(C_4k^n)\right)\\
&\leq \frac{2}{k^3} \log_{\gamma_d}(C_4k^n)
\asymp \frac{\log(k)}{k^3}\cdot
\end{align*}
Let $C_5 > 0$ be the implied constant. Combining with \eqref{Skdef} shows that
\[
S_k - S_{k + 1} \geq S_k^2 - C_1 S_k^3 + f_\psi(k) - \frac{C_1}{k^3} - \frac{C_5 \log(k)}{k^3}
\]
for all sufficiently large $k$. By Proposition \ref{propositionintegraltest}, the function
\[
x\mapsto f_\psi(x) - \frac{C_1}{x^3} - \frac{C_5 \log(x)}{x^3}
\]
is recursively integrable. By Lemma \ref{lemmaignorable}, it follows that $f_\psi\in\RR$.

\subsection{Proof of (C) \implies (D)}
As before, we will prove the contrapositive. Suppose that $f_\psi\in\RR$, and we will show that $\sup_{\R^d\butnot\Q^d} C_{\Hmax,\psi} > 0$. Fix $C_1 > 0$ large to be determined. By Lemma \ref{lemmaignorable}, the function $x\mapsto f_\psi(x) + C_1/x^3$ is recursively integrable. Thus by Proposition \ref{propositionintegraltest}, there exists a nonnegative sequence $(S_k)_{k\geq k_0}$ satisfying
\begin{equation}
\label{skdef}
S_{k + 1} = S_k - S_k^2 - C_1S_k^3 - f_\psi(k) - \frac{C_1}{k^3} \all k\geq k_0.
\end{equation}
For $k\geq k_0$, let $s_k = S_k$, $t_k = \gamma_d(1 + s_k)$, and
\[
A_k = \gamma_d^{k_0} \prod_{j = k_0}^{k - 1} t_j = \gamma_d^k \prod_{j = k_0}^{k - 1}(1 + s_j).
\]
Let $i_k = k$ (mod $d$), and consider the $d$-dimensional data progression $\Delta = (A_k,i_k)_{k = k_0}^\infty$. Since the sequence $(A_k)_{k_0}^\infty$ is increasing, Remark \ref{remarkreverseimplication} applies and we have the implication \eqref{contradictionhypothesis8} \implies \eqref{contradictionhypothesis5}. Note that if \eqref{contradictionhypothesis5} holds for all $k$ sufficiently large, then we are done, as $C_{\tmax,\Psi}(\Delta) \geq 0$ and then Lemma \ref{lemmacondition} completes the proof.

Let us proceed to demonstrate \eqref{contradictionhypothesis8}. We begin by reproving Subclaims \ref{subclaimdecreasing}, \ref{subclaimignorable1}, and \ref{subclaimignorable2} in our new context. Fix $k\in\N$. The inequality $S_{k + 1} \leq S_k$ is immediate from \eqref{skdef}. If $k$ is sufficiently large, then $f_\psi(k)\leq 1/k^2$, $k\geq C_1$, and $S_k\leq 1/C_1$, so
\[
S_k - 2S_k^2 \leq S_{k + 1} + \frac{2}{k^2}\cdot
\]
This implies that $S_k \lesssim \max(1/k^2,S_{k + 1})$, completing the proof of the analogue of Subclaim \ref{subclaimignorable1}. Finally, let $a_k = \max(1/k,S_k)$; it is immediate that $a_k \leq 1/C_1$ if $k$ is sufficiently large.

As in the proof of Claim \ref{claimtkbound} we call an expression $A$ \emph{negligible} if $|A| \lesssim a_k^3$, and write $A\sim B$ if $A - B$ is negligible. The argument following Definition \ref{definitionignorable} shows that $S_{k - j_1} S_{k - j_2} \sim S_k^2$ for all $j_1,j_2 = 0,\ldots,d - 1$. Finally, the calculations on pages \pageref{calculations1}-\pageref{calculations2} can be modified to show that
\[
\frac{1}{d - 1}f(s_{k - d + 1},\ldots,s_{k - 1}) - \frac{d}{2}f_\psi(k) - s_k \sim \frac{d}{2} C_1 \left[\frac{1}{k^3} + S_k^3\right] \geq \frac{d}{2} C_1 a_k^3.
\]
(Just multiply $C_1$ by $-1$ in each corresponding expression, and use $f_\psi(k)$ in place of $\phi_k$.) By the definition of negligibility, we have
\[
\frac{1}{d - 1}f(s_{k - d + 1},\ldots,s_{k - 1}) - \frac{d}{2}f_\psi(k) - s_k \geq \frac{d}{2}C_1 a_k^3 - C_2 a_k^3
\]
for some constant $C_2 > 0$ independent of $C_1$. Letting $C_1 = (2/d)C_2$, we have
\[
s_k \leq \frac{1}{d - 1}f(s_{k - d + 1},\ldots,s_{k - 1}) - \frac{d}{2}f_\psi(k).
\]
But since $A_k \geq \gamma_d^k$, we have $f_\psi(k) \geq \frac{2}{d\gamma_d}\Phi(A_k)$ for all sufficiently large $k$. Combining this inequality with the equation on the previous line gives \eqref{contradictionhypothesis8}, completing the proof.

\subsection{Completion of the proof of Proposition \ref{propositionHmax}}
Using the implications (C) \implies (D) \implies (B), we now complete the proof of Proposition \ref{propositionHmax}. As the implications (C) \implies (B) \implies (A) are obvious, it suffices to prove that (A) \implies (D) \implies (C). Let
\begin{align*}
\phi(q) &= q^{1/\log^3\log(q)}\\
g_\phi(x) &= \frac{2}{d\gamma_d} \frac{\log\phi(e^{\gamma_d^x})}{\gamma_d^x} = \frac{2}{d\gamma_d\log^3(\gamma_d)}\frac{1}{x^3} ,
\end{align*}
so that
\begin{align*}
f_{\phi\psi} &= f_\psi + g_\phi\\
f_{\psi/\phi} &= f_\psi - g_\phi.
\end{align*}
Since the function $g_\phi$ is ignorable, we have $f_{\phi\psi}\in\RR \Leftrightarrow f_\psi\in\RR \Leftrightarrow f_{\psi/\phi}\in\RR$. On the other hand, $\phi(q) \to \infty$ as $q\to\infty$. Thus
\[
\text{(A)} \Rightarrow \text{(C)}_{\psi = \phi\psi} \Rightarrow \text{(D)}_{\psi = \phi\psi} \Leftrightarrow \text{(D)} \Leftrightarrow \text{(D)}_{\psi = \psi/\phi} \Rightarrow \text{(B)}_{\psi = \psi/\phi} \Rightarrow \text{(C)}.
\]

\draftnewpage
\section{Open questions}
\label{sectionopenquestions}

In this paper, we consider only ``everywhere'' questions - that is, we are interested in functions $\psi$ for which $C_{H,\psi}(\xx) < \infty$ for every point $\xx\in\R^d\butnot\Q^d$. The same questions can be asked if ``every'' is replaced by ``almost every'' - with respect to Lebesgue measure or even with respect to some fractal measure. Once we know what ``almost every'' point does, it can be asked what is the Hausdorff dimension of the set of exceptions, i.e. the set of $\xx$ which behave differently from almost every point. In the case of the height function $\Hlcm$, such questions have been extensively studied. Thus, the next step in producing a Diophantine theory of the height functions $\Hmax$, $\Hmin$, and $\Hprod$ similar to that for $\Hlcm$ would be to answer the following questions:

\begin{question}[Analogue of Khinchin's theorem]
Fix $\Theta\in\{\tmax,\tmin,\tprod\}$, and let $\psi$ be a Hardy $L$-function. Must the sets $\{\xx\in\R^d : C_{H_\Theta,\psi}(\xx) = 0\}$ and $\{\xx\in\R^d : C_{H_\Theta,\psi}(\xx) < \infty\}$ be either null sets or full measure sets? If so, which one? Can the same theorem be proven with a weaker assumption than $\psi$ being a Hardy $L$-function (for example, assuming only that $\psi$ is decreasing)?
\end{question}

\begin{question}[Analogue of the Jarn\'ik--Besicovitch theorem]
With $\Theta$ and $\psi$ as before, what is the Hausdorff dimension of the set $\{\xx\in\R^d : C_{H_\Theta,\psi}(\xx) = 0\}$?
\end{question}

\begin{question}[Analogue of the Jarn\'ik--Schmidt theorem]
With $\Theta$ and $\psi$ as before, what is the Hausdorff dimension of the set $\{\xx\in\R^d : C_{H_\Theta,\psi}(\xx) > 0\}$? Does this set have large intersections with nice fractals?
\end{question}

\draftnewpage
\appendix
\section{Hardy fields}
\label{sectionhardyfields}

In this appendix we briefly recall the definition of a Hardy field and its basic properties.

Given $f:(t_0,\infty)\to\R$ and $g:(t_1,\infty)\to\R$, we write $f\sim g$ if $f(x) = g(x)$ for all sufficiently large $x$.

\begin{definition}
A \emph{Hardy field} is a collection of continuous functions\Footnote{Hardy fields are usually defined as collections of \emph{germs at infinity} rather than as collections of functions, but this distinction makes little difference in practice.} $\HH$ with the following properties:
\begin{itemize}
\item[(I)] For each $f\in\HH$, there exists $t_0\in\R$ such that $f:(t_0,\infty)\to\R$.
\item[(II)] Given $f,g\in\HH$, there exist $h_1,h_2,h_3,h_4,h_5\in\HH$ such that $f + g \sim h_1$, $f - g \sim h_2$, $fg \sim h_3$, $f/g \sim h_4$, and $f' \sim h_5$.
\end{itemize}
\end{definition}

The two primary examples of Hardy fields are the field of rational functions and the field of Hardy $L$-functions, described in the introduction. The fact that the collection of Hardy $L$-functions forms a Hardy field was proven by G. H. Hardy \cite[Theorem 1]{Hardy2}.

The most important fact about Hardy fields follows almost directly from the definition:

\begin{lemma}
\label{lemmahardy}
If $f,g\in\HH$ then either
\[
f(x)\geq g(x) \text{ for all $x$ sufficiently large}
\]
or
\[
f(x)\leq g(x) \text{ for all $x$ sufficiently large.}
\]\
\end{lemma}
\begin{proof}
Write $h\sim g - f$ for some $h\in\HH$. Then there exists a function $j\in\HH$ such that $j \sim 1/h$. It follows that $h(x)\neq 0$ for all sufficiently large $x$. Since $h$ is continuous, the conclusion follows.
\end{proof}

The following well-known lemma says that in a Hardy field, we can take the derivative of an inequality.

\begin{lemma}
\label{lemmahardy2}
If $f,g\in\HH$, $0\leq f(x)\leq g(x)$ for all $x$ sufficiently large, and $g(x)\to 0$, then
\[
|f'(x)| \leq |g'(x)| \text{ for all $x$ sufficiently large}.
\]
\end{lemma}
\begin{proof}
Write $h \sim g - f$ for some $h\in\HH$; then $0\leq f(x),h(x)$ for all $x$ sufficiently large, and $f(x),h(x)\to 0$. It follows that $f$ and $h$ are eventually decreasing, i.e. $f'(x),h'(x) \leq 0$ for all $x$ sufficiently large. Rearranging completes the proof.
\end{proof}

One last lemma which we needed in verifying the hypotheses of Lemma \ref{lemmacondition} (cf. Remark \ref{remarkcondition}):

\begin{lemma}
\label{lemmahardy3}
If $f\in\HH$ satisfies $Cx \geq f(x) \to \infty$ for some $C > 0$, then $f$ is uniformly continuous and increasing.
\end{lemma}
\begin{proof}
By Lemma \ref{lemmahardy2}, we have $|f'(x)| \leq C$, i.e. $|f'|$ is uniformly bounded. This implies that $f$ is uniformly continuous. On the other hand, by Lemma \ref{lemmahardy2} we have either $f'(x) \geq 0$ for all sufficiently large $x$, or $f'(x) \leq 0$ for all sufficiently large $x$. The second case is ruled out since $f(x)\to\infty$, so $f$ is increasing.
\end{proof}

\bibliographystyle{amsplain}

\bibliography{bibliography}

\end{document}